\documentclass[12pt,reqno]{amsart}

\usepackage{amsmath, amscd, amssymb, amsthm,amsfonts}
\usepackage{euscript, mathrsfs, latexsym,mathabx,stmaryrd}
\usepackage{color}

\usepackage{lmodern}
\usepackage[T1]{fontenc} 

\usepackage{multicol}

\ifx\pdfoutput\undefined
\usepackage{graphicx}
\else
\fi

\ifx\xetex
\usepackage{fontspec}
\usefont{EU1}{lmr}{m}{n}
\else
\fi

\usepackage{url} 
\usepackage{tikz}
\usetikzlibrary{matrix,arrows,positioning}
\tikzset{node distance=2cm, auto}

\usepackage[nodayofweek]{datetime}

\parskip=\medskipamount

\DeclareMathOperator{\Cob}{Cob}
\DeclareMathOperator{\PCob}{Pre-Cob}

\newcommand{\conj}[1]{\quad\textnormal{ #1 }\quad}

\newcommand{\inp}[1]{\ensuremath{\langle #1 \rangle}}

\newcommand{\normaltext}[1]{\textnormal{#1}}

\makeatletter
\def\imod#1{\allowbreak\mkern2.5mu({\operator@font mod}\,#1)}
\makeatother

\renewcommand{\a}{\alpha}
\renewcommand{\b}{\beta}

\newcommand{\opp}{\oplus}
\newcommand{\ott}{\otimes}

\newcommand{\s}{\sigma}

\renewcommand{\d}{\delta}

\newcommand{\bA}{\mathbf{A}}

\usepackage{bbm}

\newcommand{\CC}{\mathbb{C}}

\newcommand{\LL}{\mathbb{L}}

\newcommand{\RR}{\mathbb{R}}

\newcommand{\ZZ}{\mathbb{Z}}

\newcommand{\kk}{\mathbb{k}}

\theoremstyle{plain}

\newtheorem{thm}{Theorem}[section]
\newtheorem*{introthm}{Theorem}

\newtheorem{theorem}[thm]{Theorem}

\newtheorem{prop}[thm]{Proposition}

\newtheorem{lemma}[thm]{Lemma}

\theoremstyle{remark}

\theoremstyle{definition}

\newtheorem{example}[thm]{Example}

\newtheorem{definition}[thm]{Definition}

\newtheorem{rmk}[thm]{Remark}

\newtheorem{remark}[thm]{Remark}

\numberwithin{equation}{section}

\usepackage[shortlabels]{enumitem}

\usepackage{changes}

\usetikzlibrary{knots}
\usetikzlibrary{calc}
\usepackage[colorlinks,linkcolor=blue,citecolor=blue]{hyperref}
\usepackage{overpic}
\usepackage{enumitem}   
\usepackage[normalem]{ulem}

\usepackage{mathtools}
\usepackage{float}
\def\R{\mathbb R}
\newcommand{\ncap}{\scalebox{.8}[1]{$\bigcap$}}
\newcommand{\ncup}{\scalebox{.8}[1]{$\bigcup$}}
\newcommand{\nx}{\scalebox{.7}[1]{\scalebox{1.3}{$\bigtimes$}}}

\begin{document}

\title{An Extension of Khovanov Homology to Immersed Surface Cobordisms}
\date{10th November, 2025.}
\author[Carter]{Scott Carter}
\address{University of South Alabama, Department of Mathematics and Statistics, 411 University Boulevard North, Mobile, AL 36688-0002 USA}
\email{carter\char 64 southalabama.edu}

\author[Cooper]{Benjamin Cooper}
\address{University of Iowa, Department of Mathematics, 14 MacLean Hall, Iowa City, IA 52242-1419 USA}
\email{ben-cooper\char 64 uiowa.edu}

\author[Khovanov]{Mikhail Khovanov}
\address{Johns Hopkins University, Department of Mathematics, 404 Krieger Hall, 3400 N. Charles Street, 
Baltimore, MD 21218 USA}
\email{khovanov\char 64 jhu.edu}

\author[Krushkal]{Vyacheslav Krushkal}
\address{Department of Mathematics, University of Virginia, 
Charlottesville, VA 22904-4137 USA}
\email{krushkal\char 64 virginia.edu}

\def\JS#1{\textcolor[rgb]{0,.75,.8}{ [JS: #1]}}
\def\BC#1{\textcolor[rgb]{0,.5,1}{ [BC: #1]}}
\def\MK#1{\textcolor[rgb]{0.7,0.1,.5}{ [MK: #1]}}
\def\SK#1{\textcolor[rgb]{1,.61,0}{ [SK: #1]}}
\def\OLD#1{\textcolor[rgb]{1,.69,.4}{ [OLD STUFF: #1]}}
\def\OPT#1{\textcolor[rgb]{1,.61,0}{ [OPTIONAL: #1]}}

\newcommand{\tm}{\widetilde{m}}
\newcommand{\dga}[0]{\operatorname{-dga}}
\newcommand{\dgca}[0]{\operatorname{-dgcoa}}
\newcommand{\Om}{\Omega}
\newcommand{\pa}{\partial}

\newcommand{\Ainf}{A_\infty}
\newcommand{\varB}[1]{{\operatorname{\mathit{#1}}}}
\renewcommand{\slash}[1]{H_/(#1)}
\newcommand{\pAinf}[0]{\varB{p-\Ainf}\!}
\newcommand{\semp}{{\{\emptyset\}}}
\newcommand{\z}{z}
\newcommand{\T}{T}
\renewcommand{\d}{\delta}
\newcommand{\h}{h}
\renewcommand{\t}{t}
\newcommand{\G}{\Gamma} 
\newcommand{\A}{\Lambda} 
\renewcommand{\bA}{\bar{\Lambda}} 

\renewcommand{\kk}{k}

\newcommand{\lab}[1]{\normaltext{#1}}
\newcommand{\nt}[1]{\normaltext{#1}}
\newcommand{\Ya}{\mathcal{Y}\hspace{-.1475em}a_{2,2}}
\newcommand{\vnp}[1]{\lvert #1 \rvert}
\newcommand{\vvnp}[1]{\lvert\lvert #1 \rvert\rvert}
\newcommand{\I}{[0,1]}
\newcommand{\xto}[1]{\xrightarrow{#1}}
\newcommand{\xfrom}[1]{\xleftarrow{#1}}
\newcommand{\from}{\leftarrow}

\newcommand{\dgcat}{\normaltext{dgcat}_k}
\newcommand{\Forget}{\normaltext{Forget}}
\newcommand{\Mat}{Mat}
\newcommand{\op}{\normaltext{op}}
\newcommand{\ab}{\normaltext{ab}}

\newcommand{\Vect}{Vect}
\newcommand{\Kom}{Kom}
\newcommand{\Set}{Set}
\newcommand{\Ch}{Ch}
\newcommand{\Hom}{Hom}
\renewcommand{\d}{\delta}

\newcommand{\pre}[0]{\operatorname{pre-}}
\newcommand{\coker}[0]{\operatorname{coker}}
\newcommand{\im}[0]{\operatorname{im}}

\newcommand{\Br}{Br}

\newcommand{\BrX}{X}
\renewcommand{\b}{\beta}
\newcommand{\fun}[1]{\kk[#1]}
\newcommand{\tH}{H}
\newcommand{\HP}{HP}
\newcommand{\HC}{HC}
\newcommand{\HHH}{HHH}
\newcommand{\End}{End}

\newsavebox{\negC}
\begin{lrbox}{\negC}
\begin{tikzpicture}
\draw[ultra thick,->] (0,0) -- (0+1,0+1);
  \fill[white] (.5,.5) circle (4pt);
\draw[ultra thick,->] (0+1,0) -- (0,0+1);
\end{tikzpicture}
\end{lrbox}

\newsavebox{\unegC}
\begin{lrbox}{\unegC}
\begin{tikzpicture}
\draw[ultra thick,-] (0,0) -- (0+1,0+1);
  \fill[white] (.5,.5) circle (4pt);
\draw[ultra thick,-] (0+1,0) -- (0,0+1);
\end{tikzpicture}
\end{lrbox}

\newsavebox{\singC}
\begin{lrbox}{\singC}
\begin{tikzpicture}
\draw[ultra thick,->] (3,0) -- (4,1);
\fill[black] (3.5,.5) circle (2pt);
\draw[ultra thick,->] (4,0) -- (3,1);
\end{tikzpicture}
\end{lrbox}

\newsavebox{\usingC}
\begin{lrbox}{\usingC}
\begin{tikzpicture}
\draw[ultra thick,-] (3,0) -- (4,1);
\fill[black] (3.5,.5) circle (2pt);
\draw[ultra thick,-] (4,0) -- (3,1);
\end{tikzpicture}
\end{lrbox}

\newsavebox{\posC}
\begin{lrbox}{\posC}
\begin{tikzpicture}
\draw[ultra thick,->] (7,0) -- (6,1);
  \fill[white] (6.5,.5) circle (4pt);
\draw[ultra thick,->] (6,0) -- (7,1);
\end{tikzpicture}
\end{lrbox}

\newsavebox{\uposC}
\begin{lrbox}{\uposC}
\begin{tikzpicture}
\draw[ultra thick,-] (7,0) -- (6,1);
  \fill[white] (6.5,.5) circle (4pt);
\draw[ultra thick,-] (6,0) -- (7,1);
\end{tikzpicture}
\end{lrbox}

\newsavebox{\idTLtwo}
\begin{lrbox}{\idTLtwo}
\begin{tikzpicture}
\draw[ultra thick,black] (0,0)  .. controls (0.3,0.5) .. (0,1);
\draw[ultra thick,black] (1,0)  .. controls (0.6,0.5) .. (1,1);
\end{tikzpicture}
\end{lrbox}

\newsavebox{\eTLtwo}
\begin{lrbox}{\eTLtwo}
\begin{tikzpicture}
\draw[ultra thick,black] (0,1)  .. controls (0.5,0.6) .. (1,1);
\draw[ultra thick,black] (0,0)  .. controls (0.5,0.3) .. (1,0);
\end{tikzpicture}
\end{lrbox}

\newsavebox{\etTLtwo}
\begin{lrbox}{\etTLtwo}
\begin{tikzpicture}
\draw[ultra thick,black] (0,1)  .. controls (0.5,0.6) .. (1,1);
\fill[black] (.5,.72) circle (2pt);
\draw[ultra thick,black] (0,0)  .. controls (0.5,0.3) .. (1,0);
\end{tikzpicture}
\end{lrbox}

\newsavebox{\edTLtwo}
\begin{lrbox}{\edTLtwo}
\begin{tikzpicture}
\draw[ultra thick,black] (0,1)  .. controls (0.5,0.6) .. (1,1);
\fill[black] (.5,.2375) circle (2pt);
\draw[ultra thick,black] (0,0)  .. controls (0.5,0.3) .. (1,0);
\end{tikzpicture}
\end{lrbox}

\newcommand{\sss}{\star}
\newcommand{\C}{C}
\newcommand{\Frob}{V}
\newcommand{\bbeps}{1.618}
\newcommand{\hbbeps}{1.618/2}
\newcommand{\nbbeps}{2*1.618/3}
\newcommand{\bneps}{4*.618/3}
\newcommand{\beps}{.618}
\newcommand{\neps}{2*\beps/3}
\newcommand{\teps}{\beps/2}
\newcommand{\eps}{\beps/3}
\newcommand{\heps}{\eps/2}
\newcommand{\nshift}{10}
\newcommand{\PC}{\mathcal{C}\!{\it ob}}
\renewcommand{\I}{I}
\newcommand{\TL}{TL}
\newcommand{\Tr}{Tr}
\newcommand{\CKh}{CKh}
\newcommand{\Kh}{Kh}
\newcommand{\Amap}{A}
\newcommand{\Bmap}{B}

\begin{abstract}
We show that an oriented surface in $\RR^4$ containing double point singularities induces
a map between the Khovanov homology groups of its boundary links in a functorial
way. As part of this work, the movie moves of Carter and Saito are extended to surfaces with double points.
  \end{abstract}

\maketitle

\section{Introduction}\label{sec-intro}

From the beginning \cite[\S 1]{KhovanovJones} it was understood that a
smooth oriented surface in $\RR^3\times I$ should induce maps between
Khovanov homology groups in a functorial way. One reason for this is that,
up to a small perturbation, for a surface $\Sigma \subset \RR^3\times I$ the projection onto the time axis $I$ induces a Morse decomposition of
$\Sigma$ and assigning the maps from Khovanov's construction to the pieces produces a chain map between the chain complexes associated to the boundary links $\partial \Sigma$. 
Several years later M. Jacobsson showed
that this assignment was independent of the isotopy representative of the
surface up to sign \cite{MR2113903}. He did this by checking the movie
moves of S. Carter and M. Saito~\cite{MR1487374}.  Since
that time, simpler proofs have been found \cite{MR2171235, MR2174270} and
there have been several approaches to the sign problem \cite{MR2647055, MR2496052, MR4376719,Ca}.

Functoriality of link homology is crucial for its applications to four-dimensional topology. 
More concretely, the functoriality of Khovanov homology paved the way for
J. Rasmussen's $s$-invariant $s(K)\in \ZZ$ and his proof that it gives a lower
bound on the $4$-ball genus of a knot $K$ \cite{MR2729272}. Recently, the functoriality of Khovanov
homology became an essential ingredient of the skein lasagna invariant and 
its applications to $4$-manifold topology \cite{MWW, Renwillis}.

Our purpose in this article is to introduce a new direction of study for the functoriality of Khovanov homology.
We extend the type of surfaces which can induce maps between Khovanov homology groups by considering smooth orientable surfaces in $\RR^3\times I$ with double point singularities.
A double point singularity is locally modeled on a transverse intersection of
two planes in $\RR^4$ or the affine variety 
\begin{equation}\label{eq:zl}
Z := \{ (w,z) : wz = 0 \} \subset \CC^2\end{equation}
in a neighborhood of the origin.
The structure of $Z$ near the singular point can be described by a movie of 3-dimensional time slices of the form
\begin{equation}\label{eq-doublepointmovie}
\begin{tikzpicture}
\node at (.5,.5) {\usebox{\unegC}};
\node at (3.5,.5) {\usebox{\usingC}};
\node at (6.5,.5) {\usebox{\uposC}};
\draw [->,thick]  (1.5,.5) -- (2.5,.5);
\draw [->,thick]  (4.5,.5) -- (5.5,.5);
\end{tikzpicture}
\end{equation}
and the link $Z\cap\{ v : \vnp{v} = \epsilon \}$ 
of the double point singularity is a Hopf link in the $3$-sphere.  
We use the Khovanov homology of this Hopf link to assign a map to a smooth
oriented surface with singularities in $\RR^3\times I$ 
and prove that the homotopy class of this chain map is independent of the isotopy class of singular surface.
In order to establish our result, we extend the known movie moves to include surfaces with double point singularities and we check that our assignments satisfy these new movie moves.  Here is an informal statement of our work.

\newcommand{\Tang}{Tang}
\newcommand{\khtarget}{\mathcal{H}}

\begin{introthm}
The maps induced on the Khovanov homology of oriented links are well-defined, up to an overall sign, on isotopy classes of surface cobordisms with double points. 

These maps give rise to a functor from the category of surface cobordisms with double points between oriented links in $\R^3$ to the category of bigraded abelian groups and homogeneous maps between them, considered up to overall minus sign.
\end{introthm}

Let us state the results contained in our paper with a little more care. There is a 2-category $\Tang^4$ with objects given by finite collections $S$ of $\pm$-oriented points in the plane $\R^2$. 
 For any two such collections, $S$ and $S'$, there is a category $\Tang^4(S,S')$ whose objects are oriented tangles $T\subseteq \R^2\times [0,1]$, viewed as cobordisms from $S\times \{0\}$ to $S'\times \{1\}$ (and defined as a proper embedding of a compact oriented one-manifold into  $\R^2\times [0,1]$ rather than an equivalence class of embeddings rel boundary isotopy). 
A morphism $f: T\to T'$ between tangles is a cobordism in $4$-space. More carefully, it is an isotopy (rel boundary) class of a surface with corners standardly embedded in $\mathbb{R}^2\times [0,1]^2$ which cobounds tangles $T,T'$ in the two parallel boundary $\R^2\times [0,1]$'s and has the standard boundary $S\times [0,1]$ and $S'\times [0,1]$ on the two remaining boundary $\R^2\times [0,1]$'s. There is an extension of this category
$$\Tang^4(S,S') \subset \Tang_\times^4(S,S')$$
which has the same objects, but whose morphisms are isotopy classes of immersed surfaces with the  double point singularities locally modelled on $(Z,0)$ in Eqn. \eqref{eq:zl} above.

The next theorem is a statement about isotopy of surfaces in $4$-space. Recall that if $[f] : T\to T'$ is a morphism in $\Tang^4(S,S')$ then
any two representatives $f,f'\in[f]$ are related by a finite sequence of the Carter-Saito movie moves.

\begin{introthm}
There is an extension of the Carter-Saito movie moves from the setting of $\Tang^4(S,S')$ to the setting of $\Tang^4_\times(S,S')$. More precisely, there is a finite set of extended movie moves such that if $[f] : T\to T'$ is an isotopy class of immersed surface with double point singuarities in $\Tang^4_\times(S,S')$ then
any two representatives $f,f'\in[f]$ are related by a finite sequence of these extended moves.
\end{introthm}

The extended movie moves are catalogued, with references to illustrations, by Thm. \ref{thm: movie moves} in Section \ref{sec-moviemoves}.

Our second theorem below uses this theorem to extend the functoriality of Khovanov homology.
For context, recall that Khovanov homology associates to a cobordism $f$ a chain map $f_*$. If $f' \in [f]$ then invariance under the Carter-Saito moves, up to overall sign, implies that there is a chain homotopy $f_*\simeq \pm f'_*$. This is the key step in showing that the Khovanov construction determines a $2$-functor $\kappa : \Tang_b^4\to \widehat{\mathbb{K}}$, where $\widehat{\mathbb{K}}$ is the 2-category defined in \cite{MR2171235},  also denoted $K^\flat(\PC_V)/\{\pm 1\}$ in \S\ref{ssec-khovanov} below and throughout the paper, and see~\cite{MR2174270} for an alternative approach and a generalization to the equivariant case. (In the 2-category $\widehat{\mathbb{K}}$, 2-morphisms are defined up to overall sign.) 

Here $\Tang_b^4\subset \Tang^4$ is a full 2-subcategory of $\Tang^4$ where the objects are \emph{balanced} collections $S\subset \R^2$, i.e., with the same number of positive and negative points. Likewise one defines a full 2-subcategory  $\Tang^4_{\times,b}$ of $\Tang^4_{\times}$. 
Our extension of $\kappa$ from $\Tang_b^4$ to $\Tang^4_{\times,b}$ is defined by assigning maps to the double point cobordism and checking that the extended Carter-Saito moves hold.

\begin{introthm}
Assigning maps $A, B$ in Def. \ref{def: AB} to the double point cobordism  in Eqn. \eqref{eq-doublepointmovie} determines an extension 
$\tilde{\kappa} \colon \Tang^4_{\times,b} \to \widehat{\mathbb{K}}$
of the Khovanov $2$-functor $\kappa : \Tang^4_b\to \widehat{\mathbb{K}}$ from embedded balanced cobordisms $\Tang_b^4$ to balanced cobordisms with double points $\Tang^4_{\times,b}$.
\end{introthm}

Unlike the maps which are assigned to embedded
cobordisms, the maps we associate to surface cobordisms with double point singularities
change the homological degree. Another interesting aspect of our construction is that the maps associated with positive and negative double points are different.
These new features may be useful for applications to
topology. The study of surfaces in 4-manifolds, and particularly the question of when double points can be removed by a homotopy, is fundamental in 4-dimensional topology. Indeed, the failure of the Whitney trick in dimension 4 underlies the difference between smooth and topological 4-manifolds \cite{Donaldson, Freedman}, see also \cite{CG, Kronheimer}.
Developing a new tool to study surfaces with double points was a motivation for this work.

Related constructions appear in the recent works~\cite{ISST}, \cite{RSWWZ}. See Section \ref{sec: related} and especially items (\ref{rmk_ISST}), (\ref{rmk_RSWWZ}) in that section for a more detailed discussion.

{\bf Organization.}
\S \ref{sec-review} contains a brief review of Khovanov homology. 
\S \ref{sec-extension} introduces the maps assigned to double point singularities and checks that these maps satisfy the movie moves up to homotopy.
 \S\ref{sec-moviemoves} contains a discussion of movie moves and shows how to extend them to the setting of surfaces cobordisms with double points.

{\bf Acknowledgments.} The second author would like to thank Kevin Walker for a helpful conversation. The third author was partially supported by NSF grant DMS-2204033 and Simons
Collaboration Award 994328 ``New Structures in Low-Dimensional Topology''. The fourth author was supported in part by NSF Grant DMS-2405044.

\section{Review of Khovanov homology}\label{sec-review}

Here we review the construction of Khovanov homology and recall the duality
statements which will be used to check movie moves in \S\ref{sec-extension}.
The ideas here are equivalent to those of references such as~\cite{MR2171235, MR2174270}. Compared to \cite[\S 2.3]{MR2901969}, the $2$-category $\PC$ is analogous to $\PCob(n)$ and $\PC_V$ to $\Cob(n)$. Set $I:=[0,1]$.

There is a $2$-category $\PC$ with 
\begin{enumerate}
\item Objects given by disjoint collections of points on a line $S\subset \I$. We also use the symbol $S$ to denote its cardinality $\#S\in \ZZ_{\geq 0}$.
\item $1$-morphisms $D : S\to S'$ are  formally $q$-graded direct sums $D=\opp_{i=1}^N q^{n_i} D_i$ of $1$-manifolds $D_i : S\to S'$ in $I^2$ which bound the $0$-manifolds $S$ and $S'$ in the sense that $S\subset \I\times \{1\}$, $S'\subset I\times \{0\}$ and $\partial D_i= S \cup S'$ for $1\leq i \leq n$,
\item A $2$-morphism between $1$-morphisms $D, D' : S\to S'$ is a $\ZZ$-linear combination of surfaces $\Sigma : D \to D'$ 
 in $\I^3$. In more detail, if $D=\opp_{j=1}^N q^{n_i} D_i$ and $D'=\opp_{i=1}^M q^{m_j} D'_j$ then a map $\Sigma$ is a $M\times N$ matrix $(D_{ij})$ the entries $D_{ij}=\sum_{k} a_k \Sigma_{ij}^k$ of which are $\ZZ$-linear combinations of orientable surfaces $\Sigma_{ij}^k \subset I^3$
 which bound the $1$-manifolds $D_i$ and $D_j'$ in the sense that the boundary $\partial \Sigma_{ij}^k$ decomposes as a union of $D_i$, $D_j'$, $S\times \I$ and $S'\times \I$ along the occupied faces of the cube ($\Sigma_{ij}^k\cap (\{0,1\}\times I^2) = \emptyset$).
\end{enumerate}
Manifolds are considered up to isotopy and maps are composed along various axes by gluing and rescaling. We briefly review the setup in order to orient the reader and establish notation.

There is a product $\sqcup : \PC\times \PC \to \PC$ given by the disjoint union of cobordisms along the first axis of $\I^3$. The unit with respect to the disjoint union is the empty set $\emptyset$.

For any two sets of points $S, S'$, there is a category $\Hom(S,S')$ with objects given by $1$-morphisms $D : S \to S'$ as in $(2)$ above and maps given by $2$-morphisms in the sense of $(3)$ above.
The composition in $\PC$ gives functors 
$$\ott : \Hom(S,S') \times \Hom(S',S'') \to \Hom(S,S'').$$
In particular, if $D : S\to S'$ and $E : S'\to S''$ are $1$-manifolds then the composite $D\ott E : S\to S''$ is illustrated below.
\begin{equation}\label{eq-stacking}
\begin{tikzpicture}[scale=1]
\node (Top) at (0,1) {};
\node (TL) [left=1cm of Top] {};
\node (TR) [right=1cm of Top] {};
\node (TRR) [right=-.3cm of TR] {$S$};
\draw[thick, dashed,gray] (TL) -- (TR);
\node (TLL) [left=-.3cm of TL] {$1$};

\node (D) [below=.15cm of Top] {$D$};
\draw[thick] ($(D.north west)+(-\eps,0)$)  rectangle ($(D.south east)+(\eps,0)$);

\node (DE) [below=.02cm of D] {};
\node (DEL) [left=1cm of DE] {};
\node (DER) [right=1cm of DE] {};
\node (DERR) [right=-.3cm of DER] {$S'$};
\draw[thick, dashed,gray] (DEL) -- (DER);
\node (DELL) [left=-.3cm of DEL] {$\frac{1}{2}$};

\node (E) [below=.3cm of D] {$E$};
\draw[thick] ($(E.north west)+(-\eps,0)$)  rectangle ($(E.south east)+(\eps,0)$);

\node (Bot) [below=.15cm of E] {};
\node (BL) [left=1cm of Bot] {};
\node (BR) [right=1cm of Bot] {};
\node (BRR) [right=-.3cm of BR] {$S''$};
\draw[thick, dashed,gray] (BL) -- (BR);
\node (BLL) [left=-.3cm of BL] {$0$};

\draw[thick,->] (TRR)to node {$D$} (DERR);
\draw[thick,->] (DERR)to node {$E$} (BRR);

\draw [thick] (Top.center) --  (D);
\draw [thick] (D.south) --  (E.north);
\draw [thick] (E.south) --  (Bot.center);
\end{tikzpicture}
  \end{equation}
When $S=S'$, the identity $1$-manifold is $1_S := S\times I$.

If $\Sigma : D\to D'$ is a surface with $D,D' : S\to S'$ then the expression 
\[\chi(\Sigma) - (S+S')/2\in\ZZ
\]
is preserved by gluing. This determines the $q$-grading: if $\Sigma : q^n D\to q^m D'$ is a map between formally $q$-graded $1$-manifolds then we set
$$\deg(\Sigma) := \chi(\Sigma) - (S+S')/2 + m-n.$$
In this way, every $2$-morphism is a sum of its $q$-homogeneous components. 

We next review a duality lemma which is important in \S\ref{sec-extension}.
For more information see \cite[\S 5.1]{MR3426689} or compare \cite[\S 3.1]{MR2496052}.

\newcommand{\dthree}[1]{#1^\vee}

If $E : S\to S$ is an endomorphism then the trace $\Tr(E)$ is the closed $1$-manifold given by the quotient of $E$ which pairwise identifies top $S$-points with the bottom $S$-points:  $(x_i,0) \sim (x_i,1)$, when $S=\{x_1<x_2<\cdots <x_n\}$.
The proposition below describes how $2$-morphisms can be written as traces.

\begin{prop}\label{prop-dual}
  If $E, F :  S\to S'$ are $1$-manifold morphisms  then there is a dual $\dthree{E} : S'\to S$ and a natural isomorphism
\begin{equation}\label{eq-d3}
    \Hom(E,F) \cong q^{-(S+S')/2} \Hom(\emptyset, \Tr(F\ott \dthree{E})).
\end{equation}
The duality functor satisfies $\dthree{(\dthree{E})}\cong E$ and
$$  \dthree{(E\opp F)} \cong \dthree{E}\opp \dthree{F},\quad \dthree{(E\ott F)} \cong \dthree{F} \ott \dthree{E},\quad   \dthree{1_S} = 1_S$$
and reverses the formal $q$-grading.
  \end{prop}

The idea is that there is an isotopy of any surface 
$\Sigma : E\to F$ which pulls $E$ along
the boundary of the box to the
face occupied by $F$. This produces a closed
$1$-manifold consisting of $F$
composed with a reflection
$\dthree{E}$ of $E$. Here is a
picture of the reflection.
\begin{equation}\label{eq-reflections}
\begin{tikzpicture}
  \node (A) {$\dthree{Q}$};
  \node (B) [right=1cm of A] {$\raisebox{\depth}{\scalebox{-1}[-1]{Q}}$};
\draw[thick] ($(A.north west)+(0,0)$)  rectangle ($(A.south east)+(0,0)$);
\draw[thick] ($(B.north west)+(0,0)$)  rectangle ($(B.south east)+(0,0)$);
\draw[->,thick] ($(A.east)+(\heps,0)$) to node {$\sim$} ($(B.west)+(-\heps,0)$);
  \end{tikzpicture}
\end{equation}

If $E$ and $F$ are $1$-manifolds then $\Tr(F\ott \dthree{E})$ is a collection of circles, so it is natural to apply a Frobenius algebra or $2$-dimensional TQFT to the right hand side of Eqn.~\eqref{eq-d3}. We use this observation to describe the Khovanov homology 2-category $\PC_V$ in the next section.

\subsection{Khovanov Homology}\label{ssec-khovanov}

Bar-Natan and Khovanov $2$-category $\PC_V$ is obtained from the 2-category $\PC$ by using Eqn. \eqref{eq-d3} to require that
\begin{equation}\label{eq-localize}
  \Hom(E,F) \cong q^{-(S+S')/2} V^{\otimes \#\pi_0(F\ott \dthree{E})}.
\end{equation}  
 The circle diagram $\mathbb{S}^1$ is isomorphic to $V:=q H^*(\mathbb{S}^2)$, where 
$H^*(\mathbb{S}^2) :=  \mathbb{Z}[X]/(X^2)$
is the Frobenius algebra\footnote{Notice that $H^*(\mathbb{S}^2)$ is a graded Frobenius algebra, and $V$ is a free rank $1$ graded module over it, with $\vnp{1}_q=1\neq 0$.} graded by the degree assignments $\vnp{1}_q:=0$ and $\vnp{X}_q:=-2$. The coproduct is $\Delta(1) := X\ott 1 + 1\ott X$ and $\Delta(X) := X\ott X$. The counit is $\epsilon(X):=1$ and $\epsilon(1):=0$.

By construction the duality map descends to $\PC_V$ and the duality isomorphism \eqref{eq-d3} continues to hold.
In $\PC_V$, there are {\em delooping} isomorphisms
$$D\sqcup \mathbb{S}^1 \cong q D\opp q^{-1} D \conj{ for any 1-manifold } D\in \PC.$$

To each tangle diagram $\tau$, one can assign a chain complex $T_\tau\in \Ch(\PC_V)$
using the rule pictured below. 
For each positively
or negatively oriented crossing $\s^\pm$, the chain complex $T_{\s^\pm}$ is defined to be the
complex $\C^\pm$, which is the cone of the saddle cobordism between two
resolutions pictured below.\footnote{We abuse the notation here and omit an extra decoration on the left-hand side of these equations that would indicate the chain complex associated with a crossing.}
\begin{equation}\label{eq-crossing}
\begin{tikzpicture}[scale=0.6]
\node (W) at (0.5,.5) {\usebox{\posC}};
\node (Wt) [right=\hbbeps cm of W] {$:=$};
\node (X) [right=1.3cm of W] {\usebox{\idTLtwo}};
\draw[thick] ($(X.south west)+(\heps,0)$) -- ($(X.south east)+(-\heps,0)$);
\node at ($(X.west)+(-\eps,0)$) {$q^1$};
\node (Y) [right=\bbeps cm of X] {\usebox{\eTLtwo}};
\node at ($(Y.west)+(-\eps,0)$) {$q^2$};
\draw[->,thick]($(X.east)+(\heps,0)$) to node {} ($(Y.west)+(-\beps,0)$);
\draw[thick,dotted] ($(X.north west)+(-\bneps,\neps)$)  rectangle ($(Y.south east)+(\neps,-\neps)$);

\node (Wn) [right=4*\bbeps cm of W] {\usebox{\negC}};
\node (Wt) [right=\hbbeps cm of Wn] {$:=$};
\node (Xn) [right=1.3 cm of Wn] {\usebox{\eTLtwo}};
\node at ($(Xn.west)+(-\heps,0)$) {$q^{-2}$};
\node (Yn) [right=\bbeps cm of Xn] {\usebox{\idTLtwo}};
\node at ($(Yn.west)+(-\heps,0)$) {$q^{-1}$};
\draw[->,thick]($(Xn.east)+(\heps,0)$) to node {} ($(Yn.west)+(-\bneps,0)$);
\draw[thick] ($(Yn.south west)+(\heps,0)$) -- ($(Yn.south east)+(-\heps,0)$);
\draw[thick,dotted] ($(Xn.north west)+(-\bneps,\neps)$)  rectangle ($(Yn.south east)+(\neps,-\neps)$);
\end{tikzpicture}
  \end{equation}
The underlined $1$-morphisms are placed in homological degree zero. The duality map $\dthree{C}$ extends to chain complexes by reversing the homological $t$-degree. 
Notice that $\dthree{(C^\pm)}\cong C^\mp$. In particular, the dual $\dthree{-}$ on $1$-manifolds determines an operation on tangles 
which satisfies 
$$T_{\dthree{\tau}} \cong \dthree{T_\tau}.$$

When two tangle diagrams differ by an oriented Reidemeister move there is a degree zero chain homotopy equivalence between the associated chain complexes. In this way there is an assignment $\kappa(\tau):=[T_\tau]$ of oriented tangles to elements of the homotopy category $K^{\mathsf{b}}(\PC_V)$ of bounded chain complexes.  This assignment extends to a $2$-functor $\kappa : \Tang^4_b \to K^{\mathsf{b}}(\PC_V)$ up to sign. The 2-category $K^{\mathsf{b}}(\PC_V)/\{\pm 1\}$ is denoted by $\widehat{\mathbb{K}}$ in \cite{MR2171235}.

If $(E,d_E)$ and $(F,d_F)$ are two chain complexes in $\Ch(\PC_V)$ then there is a chain complex $(\Hom^*(E,F),\delta)$ of maps from $E$ to $F$. In degree $\ell$ an element of this chain complex is a sequence of maps $\{ f_i : E^i \to F^{i+\ell} \}_{i\in\ZZ}$. The differential is $\d f = \{ d_F f_i + (-1)^{\ell + 1} f_{i+1} d_E \}_{i\in \ZZ}$.

Eqn. \eqref{eq-d3} extends degreewise along the $\Hom^*$-construction with the functor $-^\vee$. In particular, when $E=T_\a$ and $F=T_\b$ are the chain complexes associated to tangles $\a$ and $\b$ and, by using Eqn. \eqref{eq-localize}, we obtain
\begin{align*}
  \Hom^*(T_\a,T_\b) & \cong q^{-(S+S')/2} \Hom^*(\emptyset, \Tr(T_\b\ott \dthree{T_{\a}}))\\
                    &\cong q^{-(S+S')/2} \CKh(\b\dthree{\a}).
  \end{align*}
Here $\CKh(\b\dthree{\a})$ is the standard hypercube-shaped chain complex that appears in the original construction of Khovanov homology.
So we have the following lemma.

\begin{lemma}\label{lemma-d4}
For two oriented tangles $\a, \b : S\to S'$ from $S$ points to $S'$ points, there is a isomorphism on homology,
\begin{equation}\label{eq-d4}
    H(\Hom^*(T_\a,T_\b),\delta) \cong q^{-(S+S')/2} \Kh(\b \dthree{\a}).
\end{equation}
\end{lemma}

The lemma above will be used in \S\ref{sec-extension} to study movie moves.

We conclude with what, for us, is an important example: computing the mapping space from one crossing $\C^\mp$ to the opposite crossing $\C^\pm$ results in a Hopf link.

\begin{example}\label{ex-dual}
Notice that $(\C^\pm)^\vee \cong \C^\mp$.   
If $H_\pm$ is the Khovanov homology of the Hopf link with linking number $\pm 1$ then
$$H(Hom^*(\C^\mp,\C^\pm)) \cong q^{-2} \Tr(\C^\pm\ott \dthree{(\C^{\mp})}) = q^{-2} H_\pm,$$
where
\begin{align}
H_+ &\cong \ZZ_{0,0}\opp \ZZ_{0,2}\opp \ZZ_{2,4} \opp \ZZ_{2,6}\\
H_- & \cong \ZZ_{0,0}\opp \ZZ_{0,-2}\opp \ZZ_{-2,-4}\opp \ZZ_{-2,-6}.
\end{align}
Here the symbol $\ZZ_{a,b}$ represents a $\ZZ$-summand in $(t,q)$-degree $(a,b)$.
See Fig. \ref{fig:hopflinks} below for an illustration of the homologies $H_\pm$.
\end{example}

\begin{rmk}\label{deformedhopf}
   Instead of $H^*(\mathbb{S}^2)$ in Eqn. \eqref{eq-localize} we could use the equivariant homology 
    $$H^*_{U(2)}(\mathbb{S}^2) := H^*_{U(2)}(pt)[X]/(X^2-hX-t) \conj{ where } H^*_{U(2)}(pt) := \ZZ[h,t]$$
and we follow the grading convention $\vnp{1}_q=0$ and $\vnp{X}_q=-2$, so that $\vnp{h}_q=-2$ and $\vnp{t}_q=-4$ make the quotient relation $q$-homogeneous.
    For this Frobenius algebra, the counit is determined by setting $\epsilon(1):=0$ and $\epsilon(X):=1$. The coproduct is given by $\Delta(1) :=1\ott X + X\ott 1 + h 1\ott 1$ and $\Delta(X) := X\ott X + t 1\ott 1$ \cite[Eqn. (5)]{KhovanovFrobenius}. And the unit is $\iota(1):=1$. 

    Let $R := \mathbb{Z}[h,t]$ with $\vnp{h}_q=-2$ and $\vnp{t}_q=-4$.
    Repeating \S\ref{ssec-khovanov} we assign the $R$-module $A := qH^*_{U(2)}(\mathbb{S}^2)$ to the circle $\mathbb{S}^1$ and produce a $2$-category $\PC_A$ in which the duality theorem continues to hold. Moreover, since $X^2=hX+t$, the $2$-category $\PC_V$ is obtained from $\PC_A$ by setting $h=0$ and $t=0$. If the same assignments as Eqn. \eqref{eq-crossing} are used for positive and negative crossings then there is a tangle $2$-functor $\kappa' : \Tang^4_b \to K^{\mathsf{b}}(\PC_A)/\{\pm 1\}$, see \cite{MR2174270} and \cite[Prop. 6]{KhovanovFrobenius}.
    
    Now there are isomorphisms
\begin{equation}\label{eqivdeloop}
A \mathop{\mathrel{\rightleftarrows}}^{\alpha}_{\beta} qR \opp q^{-1}R
\end{equation}
$$\conj{ where }
\alpha(z) := \left(\begin{array}{cc}\epsilon(Xz)& \epsilon(z)\end{array}\right)^T\hspace{0.25em}\textnormal{ and }\hspace{0.25em} \beta(z) := \left(\begin{array}{cc}\iota(z) & \iota(Xz)-h\iota(z)\end{array}\right).$$
    The relations $\alpha\beta=1$ and $\beta\alpha=1$ imply corresponding delooping isomorphisms among the associated surfaces in the $2$-category $\PC_A$.  By applying these delooping isomorphisms and removing acyclic subcomplexes as in \cite{FastKhovanov}, since $\mathbb{S}^1$ is identified with $A$, one computes the {\em equivariant Hopf link homologies} as
\begin{align*}
        H_+^{equiv} &= t^0q^1 A \oplus t^2q^5 A\\
        H_-^{equiv} &= t^{-2}q^{-5}A \oplus t^0 q^{-1}A.
\end{align*}
    In particular, Eqn. \eqref{eqivdeloop} shows that there is a non-canonical isomorphism of the form
    $H_\pm^{equiv} \cong H_\pm \ott_\ZZ R$.  See also Remark \ref{univfamilyrmk}.
\end{rmk}

\section{Extension of Khovanov homology to double points}\label{sec-extension}

\subsection{Chain maps associated with double points}
In order to assign a map to a singular cobordism between oriented links it
suffices to assign maps to Morse singularities, Reidemeister moves and the
double points corresponding to the passage from a positive crossing to a
negative crossing and vice versa, see Eqn. \eqref{eq-doublepointmovie}. 

\begin{definition} \label{def: AB}
For the Morse singularities and
Reidemeister moves we will use the same assignments as the existing theory
and the maps $\Amap : \C^+\to \C^-$ and $\Bmap : \C^- \to \C^+$ which are pictured in Fig. \ref{fig: mapsAB} for the double
points. 
On the right-hand side, the $B$ map consists of the identity map between the two components of the crossing complexes which are not in degree zero. It has homological degree $2$ (also called the \emph{$t$-degree}). On the left-hand side, the $A$ map is the alternating sum of dots on the two sheets corresponding to components of the crossing complexes which are not in degree zero. A dot can be understood as half of a handle, see \cite[\S 2.3]{MR2901969}. The map $A$ has homological degree $-2$.
\end{definition}
\begin{figure}[H]
    \centering
\includegraphics[width=14cm]{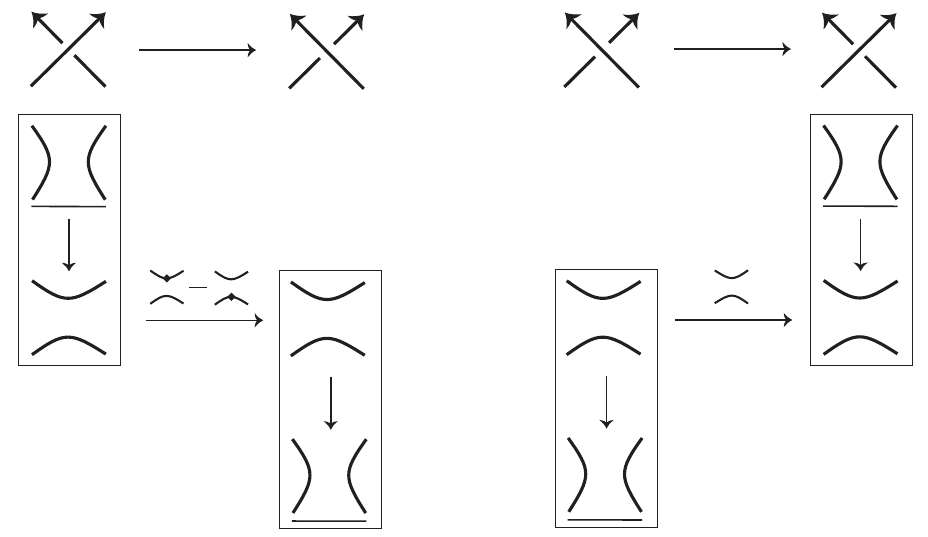}
\put(-322,215){$A$}
\put(-92,215){$B$}
\caption{The definition of the chain maps $A$, $B$ assigned to double points} \label{fig: mapsAB}
\end{figure}

 It is straightforward to check the proposition below.

\begin{prop}\label{prop-abhopf}
  The assignments $\Amap$ and $\Bmap$ are chain maps which represent homology classes in the chain complexes $\Hom^*(\C^+,\C^-)$ and $\Hom^*(\C^-, \C^+)$ corresponding to the classes in $(t,q)$-degree $\vnp{\Amap}_{t,q}=(-2,-6)$ 
  and $\vnp{\Bmap}_{t,q}=(2,4)$ of the ($q$-shifted) Hopf link homologies $q^{-2}H_-$ and $q^{-2}H_+$ respectively in Ex. \ref{ex-dual}.
\end{prop}

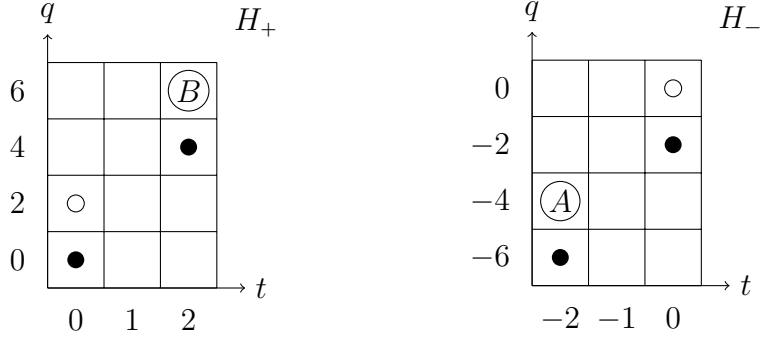
\begin{figure}
\begin{tabular}{ccc}
\begin{tikzpicture}[scale=.75]
\draw (3.7,4.7) node {$H_+$};
  \draw[->] (0,0) -- (3.5,0) node[right] {$t$};
  \draw[->] (0,0) -- (0,4.5) node[above] {$q$};
  \draw[step=1] (0,0) grid (3,4);
  \draw (0.5,-.2) node[below] {$0$};
  \draw (1.5,-.2) node[below] {$1$};
  \draw (2.5,-.2) node[below] {$2$};

  \draw (-.2,0.5) node[left] {$0$};
  \draw (-.2,1.5) node[left] {$2$};
  \draw (-.2,2.5) node[left] {$4$};
  \draw (-.2,3.5) node[left] {$6$};

\fill (0.5, 0.5) circle (.15);
  \draw (0.5, 1.5) circle (.15);

  \node[draw, circle,minimum size=.15cm, inner sep=1pt] at (2.5,3.5) {$B$};
  \fill (2.5, 2.5) circle (.15);
\end{tikzpicture}

& \hspace{.618in} &
\begin{tikzpicture}[scale=.75]
\draw (3.7,4.7) node {$H_-$};
  \draw[->] (0,0) -- (3.5,0) node[right] {$t$};
  \draw[->] (0,0) -- (0,4.5) node[above] {$q$};
  \draw[step=1] (0,0) grid (3,4);
  \draw (0.5,-.2) node[below] {$-2$};
  \draw (1.5,-.2) node[below] {$-1$};
  \draw (2.5,-.2) node[below] {$0$};

  \draw (-.2,0.5) node[left] {$-6$};
  \draw (-.2,1.5) node[left] {$-4$};
  \draw (-.2,2.5) node[left] {$-2$};
  \draw (-.2,3.5) node[left] {$0$};

\fill (0.5, 0.5) circle (.15);
  \node[draw, circle,minimum size=.15cm, inner sep=1pt] at (.5,1.5) {$A$};

  \fill (2.5, 2.5) circle (.15);
  \draw (2.5, 3.5) circle (.15);

\end{tikzpicture}
\end{tabular}
\caption{The grid shows the homology of the Hopf links $H_\pm$ from Ex. \ref{ex-dual}. 
As in Prop. \ref{prop-abhopf} the generators labelled $A$ and $B$ correspond to the maps $A$ and $B$ under the duality isomorphism. A hollow circle is a $\ZZ$-summand which is not in the image of the operation $X$ (determined by the Frobenius algebra). A filled circle is in the image of $X$. Multiplication by $X$ at a component of a link has $q$-degree $-2$. }
\label{fig:hopflinks}
\end{figure}

\begin{rmk}\label{rmk:geohopf}
    The surface $\Sigma$ produced by the Seifert algorithm from the 2-crossing planar projection of the Hopf link consists of two disks connected by two half-twisted $1$-handles. Since this surface has Euler characteristic zero, $\Sigma$ determines maps $\Sigma_* : \ZZ \to H_\pm$ of $(t,q)$-degree $(0,0)$. The image of this map is spanned by a generator of $\ZZ_{0,0}$ in $H_-$ and twice a generator of $\ZZ_{0,0}$ in $H_+$. 
 
The results in this paper imply that the class $\ZZ_{-2,-4}$ in $H_-$ corresponds to the image of the singular cobordism which is dual to the $A$ map. Decorating cobordisms with dots corresponds to the action of $X$ in Fig. \ref{fig:hopflinks}. So, abusing notation slightly,  the entire homology 
\begin{equation}\label{eq:hopfgeo}
H_- \cong \inp{\Sigma, A}\opp X\inp{\Sigma,A}
\end{equation}
can be understood in terms of cobordisms. Notice that $X\Sigma$ is a generator of the negative Hopf link homology $H_-$, but this cannot be true for the positive Hopf link $H_+$ because the $q$-degree satisfies $\vnp{X\Sigma_*(1)}_q = -2$ and the homology $H_+$ vanishes in negative $q$-degree, it follows that $X\Sigma=0$ in $H_+$. 
\end{rmk}

\begin{rmk}
Determining which classes in the Khovanov homology of a link can be represented geometrically in the manner of Eqn. \eqref{eq:hopfgeo} is the same as understanding the question of how much of Khovanov homology is {\em cobordism generated} in the sense of topological quantum field theory. This was the original motivation for the authors. Functoriality of Khovanov homology implies that for a smooth oriented surface $(\Sigma,\partial \Sigma)\subset (\R^4, \R^3)$ bounding a link $\partial \Sigma$, there is a map 
   $$\Sigma_* : \ZZ \to \Kh^{0,\chi(\Sigma)}(\partial \Sigma),$$ see \cite[\S 6.3]{KhovanovJones}. 
   Since the homological $t$-degree of the right-hand side must be zero and, at least phenomenologically, Khovanov homology is concentrated along a diagonal in the $(t,q)$-plane \cite{KhovanovPatterns}, one should expect very little of the homology to be geometric. The extension $\Tang^4_b\subset \Tang^4_{\times,b}$ considered in this paper can be considered the simplest non-trivial extension of the theory along homology classes which are {\em not} represented geometrically. In our extension, it is no longer necessary for geometric homology classes to have $t$-degree $0$. It seems interesting to ask, which classes in $\Kh^{i,j}(L)$ are now geometric? Answers to this question may help to shed light on the nature of Khovanov homology as a tool for the study of low-dimensional topology. 
\end{rmk}

\begin{rmk}
In the skein lasagna module theory~\cite{MWW} 
one considers a disjoint collection of input 4-balls $\mathbb{D}^4$ in a 4-manifold $M$ and surfaces $\Sigma\subset M$ with boundary in $\partial \mathbb{D}^4$. Additionally, these skein lasagna fillings are decorated with Khovanov-Rozansky homology classes of the links $\partial\Sigma$ in the 3-spheres $\partial \mathbb{D}^4$. The results of this paper may be interpreted as considering input balls with Hopf links $H_{\pm}$ in their boundary, decorated with  generating classes in cohomological degrees $\pm 2$. See also \cite[Example 3.7]{MWW1}.

\end{rmk}

\subsection{Related work in the literature} \label{sec: related} Here we discuss work of other authors related to the results in our paper.

\begin{enumerate}
\item
L. Weng introduced assignments for framed singular cobordisms \cite{MR2995926}. 
Interpreting these assignments in the oriented setting, the map $c_1 : \C^+\to \C^-$ in~\cite{MR2995926} has $t$-degree zero, sending the 0-resolution of $\C^+$ isomorphically onto the 1-resolution of $\C^-$, and sending the 1-resolution to $0$. The map $c_2 : \C^-\to \C^+$ in~\cite{MR2995926} is the map $\Bmap$ above.

\item \label{rmk_ISST} While preparing this paper for publication, we learned from T.~Sano about a related extension of tangle cobordism invariants to double point singularities, by H.~Imori, T.~Sano, K.~Sato, and M.~Taniguchi~\cite{ISST}.

\item \label{rmk_RSWWZ} 
After posting the first version of this paper, we learned from Q. Ren that a related construction recently appeared in \cite[Appendix A]{RSWWZ}, which can be used to establish full functoriality (not just up to sign) in the context of {\em singular foams}. In more detail, that reference considers singular $\mathfrak{gl}_2$ foams where double points are allowed between 1- and 2-labeled faces or between 2-labeled faces. However, it is pointed out in \cite[Remark A.2]{RSWWZ} that transverse double points could be allowed between 1-labeled faces as well.

For convenience of the reader, we briefly summarize this construction.
The invariant of singular $\mathfrak{gl}_2$ foams is defined in Section A.2.3 in \cite{RSWWZ}: remove a small 4-ball neighborhood of each double point, and isotope it along an arc to a boundary 3-sphere (dragging the rest of the foam by an isotopy) to get a collection of split Hopf links in $S^3$. Then decorate the Hopf links by some element of their $\mathfrak{gl}_2$ homology, and the resulting foam cobordism defines a chain map between the chain complexes of the original links. The key fact is that the homotopy class of this chain map is independent of the choice of arcs, and this is shown analogously to the proof of \cite[Theorem 5.2]{MWW}. This proof crucially relies on the sweep-around property \cite[Theorems 1.1, 3.3]{MWW}. The resulting chain map, corresponding to a movie of diagrams representing an immersed foam cobordism, is computed in \cite[Lemma A.10]{RSWWZ}.

We remark that both perspectives are useful: the approach using movie moves for singular surfaces considered in this paper, and the skein lasagna -- inspired approach of \cite{RSWWZ}. For example, the Khovanov-Lipshitz-Sarkar stable homotopy refinement is functorial, up to sign, with respect to embedded surface cobordisms \cite{LLS}. That theory is not known to satisfy the sweep-around property, and therefore an approach to extending it to immersed surfaces would rely on checking movie moves established in this paper.

\item
In a series of papers~\cite{IY2,Yo,IY1} N.~Ito and J.~Yoshida introduce and study homology of singular links, by defining the complex for a singular crossing to be the cone of the map $A$ above. One of their goals is to categorify Vassiliev invariants of links. In particular, in the latest paper~\cite{IY1}, the authors show the categorical analogue of the $4T$ relation on weight spaces, for the singular link Khovanov homology complexes built via their construction. Our use of the map $A$, to extend homology from embedded to singular embedded surfaces, is different from theirs. 

\item 

A map on Lee homology associated with double points was considered in \cite[Section 3]{AlishahiDowlin}. After posting our paper we learned that the functoriality of an analogue of this map for Khovanov homology with respect to immersed surface cobordisms, over ${\mathbb Z}/2$, has also been considered in \cite{Gujral}. To relate this map to the terminology of our paper, it is the sum of the $A$ and $B$ maps in Definition \ref{def: AB}.
\end{enumerate}

\newcommand{\UL}{W}
\renewcommand{\LL}{X}
\newcommand{\UR}{Y}
\newcommand{\LR}{Z}

\subsection{Checking the new movie moves}\label{ssec-moviemoves}

As mentioned in the introduction, Khovanov homology is known to satisfy the
movie moves for smooth oriented surfaces up to sign. In order to show that the chain homotopy
class of a map assigned to a surface with double points is independent of
isotopy we must check the new movie moves which contain double point singularities.

The new movie moves involving double points are enumerated by Theorem \ref{thm: movie moves} in Section \ref{sec-moviemoves}.
Movies MM16 and MM17 are the only ones which require non-trivial verifications. The lemmas below are introduced in order to simplify exposition later.
Lemma \ref{lem-movie} 
gives criteria in which a diagram of chain complexes commutes up to homotopy (and sign). 

\begin{lemma}\label{lem-movie}
Fix a $(t,q)$-bidegree $\sss$ and suppose that we are given a (not-necessarily commutative) diagram of chain complexes:
$$\begin{tikzpicture}
    \node (X) {$\UL$};
    \node (BX) [below=1cm of X] {$\LL$};
    \node (Y) [right=1cm of X] {$\UR$};
    \node (BY) [below=1cm of Y] {$\LR.$};
    \draw[->](X) to node [swap] {$\alpha$} (BX);
    \draw[->](Y) to node {$\gamma$} (BY);
    \draw[->](X) to node {$\beta$} (Y);
    \draw[->](BX) to node {$\eta$} (BY);
    \end{tikzpicture}$$
Then either set of conditions (1) or (2) below imply that this diagram commutes up to homotopy and sign.
\begin{enumerate}
\item  \begin{enumerate}
        \item $\alpha$, $\beta$ are degree zero homotopy equivalences
        \item $H_\sss(\Hom^*(\LL,\LR)) \cong 0 \textnormal{ or } \ZZ$
        \item $[\eta]$ and $[\gamma]$ generate $H_\sss(\Hom^*(\LL,\LR))$ and $H_\sss(\Hom^*(\UR,\LR))$ respectively
       \end{enumerate}
\item      \begin{enumerate}
       \item $\alpha$, $\gamma$ are degree zero homotopy equivalences
        \item $H_\sss(\Hom^*(\UL,\UR)) \cong 0 \textnormal{ or } \ZZ$
       \item $[\beta]$ and $[\eta]$ generate $H_\sss(\Hom^*(\UL,\UR))$ and $H_\sss(\Hom^*(\LL,\LR))$ respectively
      \end{enumerate}

    \end{enumerate}
  \end{lemma}

\begin{remark}
In each case, condition $(b)$ and condition $(a)$ combine to imply that other $\Hom$-complexes have homology which is isomorphic to $0$ or $\ZZ$.
  \end{remark}

\begin{proof}
{\em For (1): } Consider
\newcommand{\dv}{.25cm}
$$\begin{tikzpicture}
\node (X) {$H_\sss(\Hom^*(\UR,\LR))$};
\node (Y) [right=1cm of X] {$H_\sss(\Hom^*(\UL,\LR))$};
\node (Z) [right=1cm of Y] {$H_\sss(\Hom^*(\LL,\LR))$};
\draw[->] (X) to node {$\beta^*$}  (Y);
\draw[->] (Z) to node [swap] {$\alpha^*$}  (Y);

\node (Xd) [below=\dv of X] {$\gamma : \UR\to \LR$};
\node (Yd) [below=\dv of Y] {$\gamma \beta : \UL \to \LR$};
\node (Zd) [below=\dv of Z] {};
\draw[|->] (Xd) to node {} (Yd);

\node (Ydd) [below=\dv of Yd] {$\eta\alpha : \UL \to \LR$};
\node (Zdd) [right=1.9cm of Ydd] {$\eta : \LL \to \LR$};
\draw[|->] (Zdd) to node {} (Ydd);
    \end{tikzpicture}$$

Since $\alpha$ and $\beta$ are degree zero homotopy equivalences, the maps $\alpha^*$ and $\beta^*$ are isomorphisms between the homology groups pictured above. There are two cases,

If $H_\sss(\Hom^*(\LL,\LR)) \cong \ZZ$ then $H_\sss(\Hom^*(\UL,\LR)) \cong \ZZ$ via $\alpha^*$ and $H_\sss(\Hom^*(\UR,\LR))\cong \ZZ$ via $(\beta^*)^{-1}\alpha^*$. By assumption $[\eta]$ generates  $H_\sss(\Hom^*(\LL,\LR))$ and $[\gamma]$ generates $H_\sss(\Hom^*(\UR,\LR))$. Since $Aut(\ZZ)=\{\pm 1_\ZZ \}$, $\alpha^*([\eta])=\pm\beta^*([\gamma])$ or $\eta\alpha\simeq \pm \gamma\beta$.

If $H_\sss(\Hom^*(\LL,\LR)) \cong 0$ then the same argument shows that all of the homology groups are zero and $\eta\alpha\simeq \gamma\beta$.

{\em For (2): } Same proof as (1) after replacing $\beta^*$ and $\alpha^*$ with $\gamma_*$ and $\alpha^*$.

  \end{proof}

Recall that if $\beta\in \Br_n$ is a braid then, given any orientation of the strands, using the chain complexes associated to the crossings in Eqn. \eqref{eq-crossing}, gives a chain complex 
$$T_\beta := \C^\pm_{i_1} \ott \C^\pm_{i_2} \ott \cdots \ott \C^\pm_{i_n} \conj{ where } \beta = \s_{i_1}^\pm\s_{i_2}^\pm\cdots \s_{i_n}^\pm $$
associated to the braid $\beta$ as well as an inverse chain complex
$T^{-1}_\beta := T_{\beta^{-1}}$. This is an inverse in the sense that there are canonical homotopy equivalences of the form
\begin{equation}\label{eq-inverses}
  T_\beta \ott T_\beta^{-1} \simeq 1_n \conj{ and } T_\beta^{-1} \ott T_\beta \simeq 1_n.
  \end{equation}
In other words, $T_\beta$ is invertible in the homotopy category $K^{\mathsf{b}}(\PC_V)$.

Recall for (2) below that $\Frob := qH^*(\mathbb{S}^2)$ is associated to the circle $\mathbb{S}^1$ in $\PC_V$.

\begin{lemma}\label{lem-invert}
Let $\UL$ and $\LL$ be chain complexes in $\PC_V$. Then for any braid $\b\in\Br_n$ there are homotopy equivalences of mapping complexes:
\begin{enumerate}
 \item $r_\b : \Hom^*(\UL,\LL) \xto{\sim} \Hom^*(\UL\ott T_\b, \LL\ott T_\b)$ and $\ell_\b : \Hom^*(\UL,\LL) \xto{\sim} \Hom^*(T_\b \ott \UL, T_\b \ott \LL)$ 
  \item $\Hom^*(\UL\sqcup 1, \LL\sqcup 1) \xto{\sim} \Hom^*(\UL,\LL)\ott \Frob$.
\end{enumerate}
  \end{lemma}
\begin{proof}
{\em For (1):}   The map $r_\b(f) := f \ott 1_{T_\beta}$  has a homotopy inverse $b_\b(g) := g \ott 1_{T_\b^{-1}}$ because composing gives
$$\Hom^*(\UL,\LL) \xto{r_\b} \Hom^*(\UL\ott T_\b, \LL\ott T_\b) \xto{b_\b} \Hom^*(\UL\ott T_\b\ott T_\b^{-1}, \LL\ott T_\b \ott T_\b^{-1})$$
and the homotopy equivalences in Eqn. \eqref{eq-inverses} induce a natural equivalence between the righthand side and the lefthand side. The argument is parallel for the map $\ell_\b$.

{\em For (2):} Any interaction between a cobordism with a disjoint sheet can be disentangled using delooping isomorphism. Alternatively, this follows immediately from Lemma \ref{lemma-d4}.
  \end{proof}

We are now prepared to discuss the movie moves listed in Theorem~\ref{thm: movie moves} in Section~\ref{sec-moviemoves}. 

\begin{proof}{({\bf MM16})}\label{proof-MM16}
Consider movie move \#16: passing a node over a type-II move.

\[\includegraphics[width=.4\textwidth]{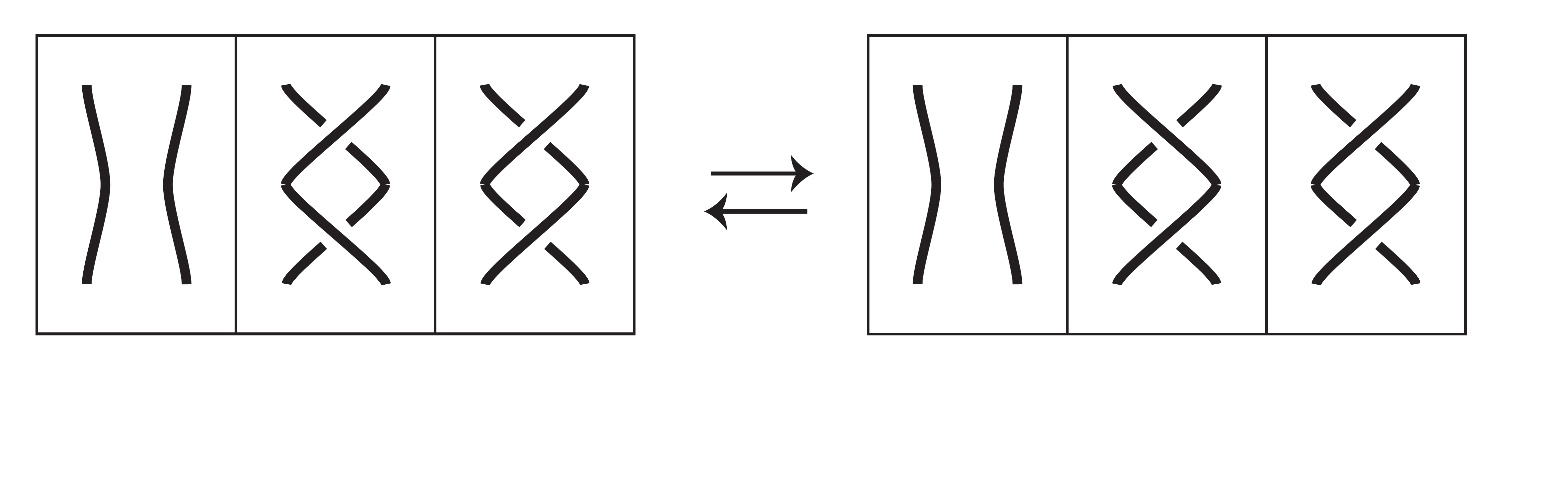}\]

For the move pictured above, each strand can be oriented, either to point up or to point down, so there are four oriented versions of this move. They involve either a positive or negative double point, resulting in one of the maps $A, B$.
Let's consider the case of both strands pointing up or both strands pointing down; in this case the diagram can be written as
$$\begin{tikzpicture}
    \node (X) {$1\ott 1$};
    \node (BX) [below=1cm of X] {$\C^+\ott \C^-$};
    \node (Y) [right=1cm of X] {$\C^-\ott \C^+$};
    \node (BY) [below=1cm of Y] {$\C^+\ott \C^+.$};
    \draw[->](X) to node [swap] {$R2$} (BX);
    \draw[->](Y) to node {$\Bmap\ott 1$} (BY);
    \draw[->](X) to node {$R2$} (Y);
    \draw[->](BX) to node {$1\ott \Bmap$} (BY);
    \end{tikzpicture}$$
Now using part (1) of Lem. \ref{lem-movie} with $(t,q)$-bidegree $\sss := \vnp{\Bmap}_{t,q} = (2,4)$ we verify the assumptions
\begin{enumerate}[(a)]
\item both maps $\alpha$ and $\beta$ are induced by Reidemeister 2 moves which are degree zero homotopy equivalences. 
\item the duality lemma \eqref{lemma-d4} gives
$$H_\sss(\Hom^*(\C^+\ott \C^-,\C^+\ott \C^+)) \cong q^{-2} H_+ \cong \ZZ$$
where $H_+$ is the positive Hopf link homology, see Ex. \ref{ex-dual}.
\item Lem. \ref{lem-invert} tells us that there is a homotopy equivalence $r_{\s}$ such that  $r_{\s}(\Bmap) = \Bmap \ott 1$. It follows that
$$H_\sss\Hom^*(\C^+\ott \C^-,\C^+\ott \C^+) \cong \ZZ\inp{1\ott \Bmap}.$$
In the same way, the isomorphism $\ell_{\s^{-1}}(\Bmap)= 1\ott \Bmap$ shows that the map $1\ott \Bmap$ generates.
\end{enumerate}

For the other two orientation choices (one strand in MM16 pointing up, the other one pointing down) the proof involves the map $A$ and is analogous. In more detail, the diagram reads 
$$\begin{tikzpicture}
    \node (X) {$1\ott 1$};
    \node (BX) [below=1cm of X] {$\C^-\ott \C^+$};
    \node (Y) [right=1cm of X] {$\C^+\ott \C^-$};
    \node (BY) [below=1cm of Y] {$\C^-\ott \C^-.$};
    \draw[->](X) to node [swap] {$R2$} (BX);
    \draw[->](Y) to node {$\Amap\ott 1$} (BY);
    \draw[->](X) to node {$R2$} (Y);
    \draw[->](BX) to node {$1\ott \Amap$} (BY);
    \end{tikzpicture}$$
The proof proceeds as in the previous case, using Lem. \ref{lem-movie} and $(t,q)$-bidegree $\sss := \vnp{\Amap}_{t,q} = (-2,-6)$ corresponding to a generator of $q^{-2} H_-$.
\end{proof}

\begin{proof}{({\bf MM17})}
  Consider movie move \#17: passing a node through a type-III move,

\[\includegraphics[width=.4\textwidth]{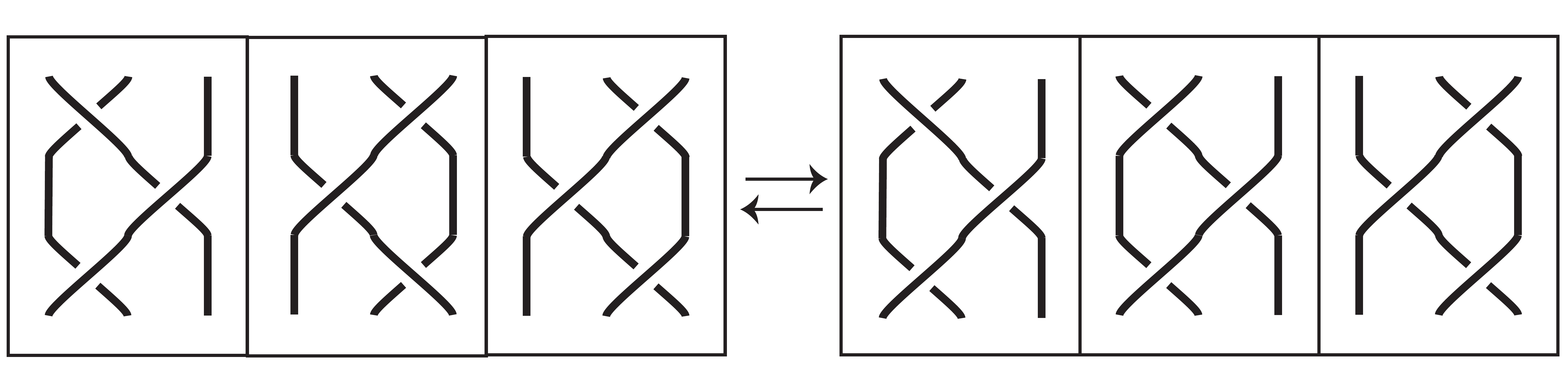}\]

For the move pictured above, each strand can be oriented in either direction, giving a total of eight cases. As in the previous proof, the map $A$ is involved in half of the cases and the map $B$ in the other half. 
When all of the strands are pointing upward we get the diagram below.
\begin{equation} \label{eq:MM17diagram}
\begin{tikzpicture}
    \node (X) {$\C^-_1\ott T$};
    \node (BX) [below=1cm of X] {$T\ott \C_2^-$};
    \node (Y) [right=1cm of X] {$\C_1^+\ott T$};
    \node (BY) [below=1cm of Y] {$T\ott \C^+_2$};
    \draw[->](X) to node [swap] {$R3$} (BX);
    \draw[->](Y) to node {$R3$} (BY);
    \draw[->](X) to node {$\Bmap\ott 1$} (Y);
    \draw[->](BX) to node {$1\ott \Bmap$} (BY);
    \end{tikzpicture},
    \end{equation}
where $T:= \C^+_2\ott \C^+_1$. Without orientations, this is the same as the braids pictured below. 
 \begin{figure}[ht]
 \centering{
\includegraphics[height=5cm]{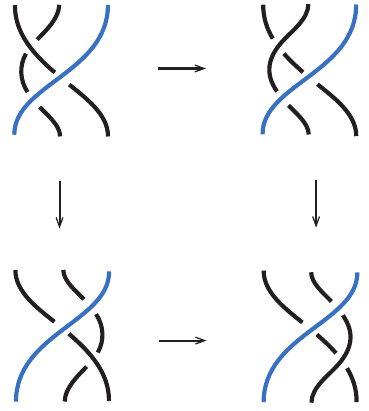}
{\small
\put(-77,125){$B\otimes 1$}
\put(-126,70){$R3$}
\put(-15,70){$R3$}
\put(-77,30){$1\otimes B$}
}}
\label{fig:CD} 
\end{figure}

Now using part (2) of Lem. \ref{lem-movie} with $(t,q)$-bidegree $\sss := \vnp{\Bmap}_{t,q} = (2,4)$ we verify the assumptions
\begin{enumerate}[(a)]
\item the maps $\alpha$ and $\gamma$ are induced by Reidemeister 3 moves, so they are degree zero homotopy equivalences. 
\item the duality Eqn. \eqref{eq-d4} identifies the horizontal mapping space with the trace of a braid that is equivalent to the positive Hopf link and an unknot $H_+\sqcup \mathbb{S}^1$. The trace described here is pictured below.

\begin{figure}[H]
\hspace{.75in}\includegraphics[height=4cm]{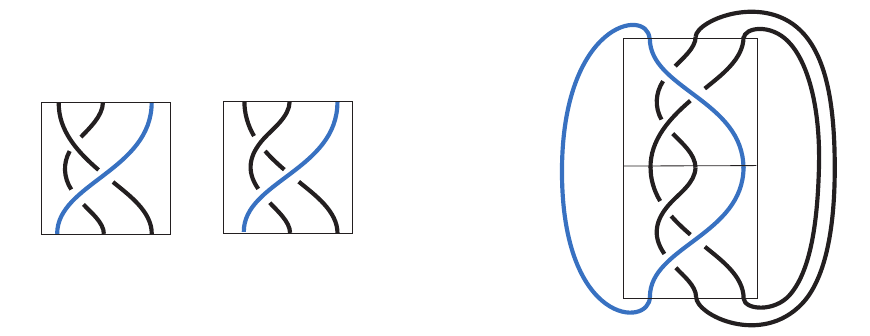}
{\large 
\put(-327,53){${\rm Hom}\Bigg( $}
\put(-229,38){$,$}
\put(-171,53){$\Bigg) $}
\put(-158,52){$\cong\; {\rm CKh}$}
}
\end{figure}

Here is a line-by-line proof:
\begin{align*}
  \Hom^*(\C_1^-\ott T,\C_1^+\ott T) &= \Hom^*(\C_1^-\ott \C^+_2\ott \C^+_1,\C_1^+\ott \C^+_2\ott \C^+_1)\\
  &\cong q^{-3} \Tr((\C_1^-\ott \C^+_2\ott \C^+_1)^\vee \ott \C_1^+\ott \C^+_2\ott \C^+_1)\\
&\cong q^{-3} \Tr(\C_1^-\ott \C^-_2\ott \C^+_1\ott \C_1^+\ott \C^+_2\ott \C^+_1)\\
&\simeq q^{-3} \Tr(\C_1^-\ott \C^-_2\ott \C^+_1\ott (\C_1^+\ott \C^+_2\ott \C^+_1))\\
&\simeq q^{-3} \Tr(\C_1^-\ott \C^-_2\ott (\C^+_1\ott \C_2^+\ott \C^+_1)\ott \C^+_2)\\
&\simeq q^{-3} \Tr(\C_1^-\ott (\C^-_2\ott \C^+_2)\ott \C_1^+\ott \C^+_2\ott \C^+_2)\\
&\simeq q^{-3} \Tr((\C_1^-\ott 1\ott \C_1^+)\ott \C^+_2\ott \C^+_2)\\
  &\simeq q^{-3} \Tr(\C^+_2\ott \C^+_2)\\
  &\cong q^{-3}(q+q^{-1})H_+.
  \end{align*}
The additional unknot $\mathbb{S}^1$ contributes the factor $q+q^{-1}$ to the $\Hom$-space. This is a consequence of computing the trace of $(\sigma_2)^2\in \Br_3$ as a braid of index $3$ and the delooping isomorphism, see the illustration above. Using the homology $H_+$ from Example~\ref{ex-dual} shows that 
$$H_\sss\Hom^*(\C_1^-\ott T,\C_1^+\ott T)\cong \ZZ$$
\item By Lemma \ref{lem-invert} there are isomorphisms $r_\beta$ and $\ell_\beta$, with $\beta=\s_2\s_1$, taking the map $\Bmap$ to $\Bmap\ott 1$ and $1\ott \Bmap$. This shows that these maps generate homologies of their respective mapping spaces in degree $\sss$.
\end{enumerate}

This completes the argument for half of the orientations. For the other half, the top and bottom maps in diagram \eqref{eq:MM17diagram} are replaced with $\Amap\ott 1$ and $1\ott \Amap$ respectively, and the proof is completed using the $(t,q)$-bidegree $\sss := \vnp{\Amap}_{t,q} = (-2,-6)$ and the link $H_-\sqcup \mathbb{S}^1$.

Just like there are several version of the third Reidemeister move, there are several versions of the movie move MM17. The proof for other versions is directly analogous to the one given above.
  \end{proof}

\begin{proof}{({\bf MM18})}
In the context of the Khovanov construction, the horizontal arrows are identity maps so the diagram commutes.
  \end{proof}

\begin{proof}{({\bf Thm. \ref{thm: movie moves} (5)})}
Movie moves involving far-commutativity commute because the maps $A$ and $B$ are applied to the same diagrams after planar isotopies which induce identity maps in the setting of $\PC_V$.
  \end{proof}

\begin{rmk}\label{univfamilyrmk}
Here are two observations about gradings in the equivariant setting $\PC_A$ of Remark \ref{deformedhopf}.
\begin{enumerate}
    \item The only monomial $h^i t^j$ in the ground ring $R$ with non-negative $q$-degree is the identity element $1\in R$ in $q$-degree $0$.
    \item The maps $A$ and $B$ as elements of $q^{-2}H_-$ and $q^{-2}H_+$ generate the class of largest $q$-degree within their respective $t$-degrees.
\end{enumerate}
Together these imply that the maps $A$ and $B$ are the only generating classes within their respective $(t,q)$-degrees of the $\Hom$-complexes associated to the homologies $q^{-2}H^{equiv}_\pm \cong q^{-2}H_{\pm} \ott R$ in $\PC_A$. So there are isomorphisms: 
$$H_{2,4}(\Hom^\ast(C^+,C^-)) \cong \ZZ\inp{B} \conj{ and } H_{-2,-6}(\Hom^\ast(C^-,C^+)) \cong \ZZ\inp{A}.$$
These equations and the two observations suffice to amend the arguments above
for the $\PC_A$ theory.  We conclude that there is a corresponding extension of the 2-functor $\kappa' : \Tang^4_b\to K^{\mathsf{b}}(\PC_A)/\{\pm 1\}$ to a 2-functor $\tilde{\kappa}' : \Tang^4_\times \to K^{\mathsf{b}}(\PC_A)/\{\pm 1\}$ which assigns the maps $A$ and $B$ to double point singularities. 
\end{rmk}

\section{Movie moves for immersed surface cobordisms} \label{sec-moviemoves}\label{sec:mm}

The Carter-Saito movie moves \cite[2.6]{CSKnottedSurfaces}, \cite{CRS} provide a combinatorial description of isotopies of surfaces embedded in 4-space. The movies involve planar projections of links in $\R^3\times\{ s\}\subset \R^3\times\R=\R^4$ which are cross-sections of the surface. We  refer to these moves according to their enumeration MM1 - MM15, cf. \cite[Figures 11-13]{MR2174270}. Movie moves for foams were also studied in \cite{HoelWalker}. Those authors complete the list  that was proposed in~\cite{JSC}. In addition,~\cite{Borodzik} study movie moves in the presence of symmetries on the knotted surfaces.
   
In this section we formulate and prove an extension of the movie moves to {\em immersed} surfaces in $\R^4$. 
To set up the notation, consider a properly immersed oriented compact surface $F$ in $\R^3 \times  [0, 1]$ and proper isotopies thereof. An immersion is of the form:
\[ (F; \partial_0 F \sqcup \partial_1 F) \looparrowright (\R^3 \times [0, 1]; \R^3 \times \{0\} \sqcup \R^3 \times \{1\}),\]
where one or both of $\partial_0 F, \partial_1 F$ may be empty. 

As discussed in the introduction, the singularities of such surfaces consist of double points, which we will also refer to as {\em nodes}. There are finitely many nodes, and they are in the interior of the surface.
The boundaries
$\partial_0 F, \partial_1 F$ are classical links embedded in  $\R^3 \times \{j\}$, $j = 0, 1$. Our goal is to analyze {\em isotopies between surfaces with nodes}, i.e. the restriction to $F$ of ambient proper isotopies of $\R^3\times [0,1]$. In particular, the nodes stay in the interior of $\R^3\times [0,1]$ and their number remains fixed during an isotopy.

The tool that will be used in our analysis is a (retinal) {\em chart}, a certain graph which arises from a planar projection of $F$, see \cite[1.5]{CSKnottedSurfaces}, \cite[Section 3.2]{CRS}. The charts of properly isotopic embedded surfaces are related by moves discussed in \cite[Theorem 2.17]{CSKnottedSurfaces}. We caution that the choice of vertical and horizontal axes in our diagrams differs from that in \cite{CSKnottedSurfaces}, and our charts have additional decorations which we describe in Section \ref{sec:charts}. 

The structure of charts of surfaces embedded in 4-space and the moves on charts encoding isotopies of such surfaces were deduced in \cite[Section 4]{CRS} from the analysis of singularities of generic maps of surfaces into $\R^2$ in \cite{Goryunov, Rieger, West}. The same type of analysis applies in our context, where a generic map of a surface into $\R^2$ is obtained starting from a surface with double points, rather than from an embedded surface in $\R^4$, and projecting onto a plane (see Section \ref{sec:charts}). 

We start by setting up the notation for encoding the links arising as the cross-sections of $F$.

{\bf Convention and terminology.} The interval factor of $\R^3\times [0, 1]$ is parametrized by the variable $s$. The coordinates of $\R^3$ are denoted $x, y, z$, and the crossings of a link in $\R^3$ are defined with respect to the $z$ coordinate, that is links are projected onto the $xy$-plane. The $y$-coordinate will serve as the height function in the $xy$-plane. The terms {\em type-I, II, or III} will refer to Reidemeister moves of a given type.

\subsection{Encoding the cross-sections of a surface} Consider an immersion $F\looparrowright \R^3 \times [0, 1]$ which is in general position with respect to the projection $\R^3\times [0,1]\to [0,1]$. In more detail, the critical points of the composition 
\[ F\looparrowright \R^3 \times [0, 1] \overset{p}{\longrightarrow} [0,1] \]
are non-degenerate, have distinct values in $[0,1]$, and are disjoint from the double points of the immersion.

Given $s\in [0,1]$, the intersection $F\cap (\R^3\times \{ s\})$ is called a {\em still} of a movie. The surface $F$ will be assumed to be oriented, and the stills are all oriented accordingly as well. Consider the projection onto the first factor, $\R^3\times [0,1]\to \RR^3$, and a further projection $\pi\colon \R^3\to \R^2$ onto the $xy$-plane. 

We can pick a sequence of generic values $\{ s_i\}$ of the parameter $s$ so that
the projections $\pi(F\cap (\R^3\times \{ s_i\}))$ and $\pi(F\cap (\R^3\times \{ s_{i+1}\}))$ of two successive stills differ by a birth, death, saddle, Reidemeister move, $\psi$-move (illustrated in Figure \ref{fig:A}),
crossing change (node), cusp, or critical exchange (exchange of the heights of critical points). 

The projection of a still onto the $xy$-plane is
an immersed curve that has isolated transverse double points. The critical points with respect to the height function $y$ of the projection of the still
have distinct $y$-coordinates, and these are distinct from
the $y$-coordinates of the crossing points which also occur at distinct levels.

The {\em bookmark code} is a way to combinatorially encode the projection onto the $xy$-plane of
a still of a movie; a similar method will be used to label the edges of the chart graph defined in Section \ref{sec:charts}. To put this in the context of our goal, Theorem \ref{thm: movie moves} below, the isotopies of immersed surfaces will first be encoded using moves on charts which are then translated to movie moves on stills using the bookmark code.

Figure \ref{fig:I} indicates how the projection of an oriented link can be written in terms of symbols corresponding to crossings and to optima (maxima and minima).  The fonts $\nx, \overline{\nx}, \ncap, \ncup$
are adorned with dots at the NE, SE, SW, or
NW directions to indicate that directional arrows, corresponding to the link orientation, emerge from such points.
The projection $F \cap (\R^3 \times \{s\}) \to \R$ onto the $y$-coordinate is a
Morse function for the link. Since the crossings are assumed to lie at different
vertical $(y)$ levels, there is a bookmarked word that can be used to describe a
knot diagram. The word is constructed from top to bottom. 
\begin{figure}[H]
\includegraphics[width=.3\textwidth]{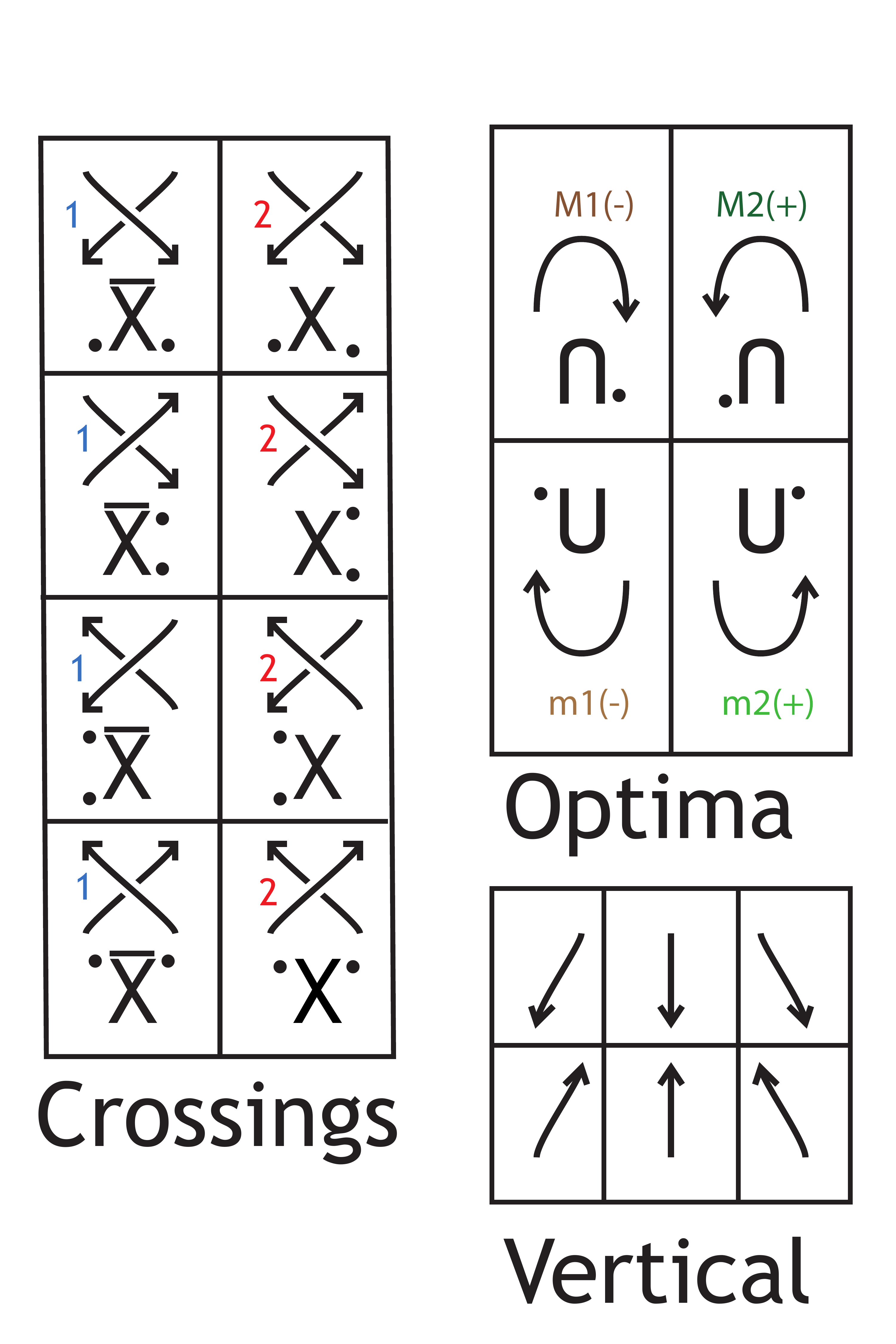}
\medskip
\caption{Types of crossings and critical points (optima) of links, and the corresponding bookmark labels}
\label{fig:I} 
\end{figure}

Each symbol is labeled with a pair of integers that indicate the number of vertical
segments to the left and the number of vertical segments to the right of a given crossing or critical point.
In general, the bookmarked word always starts with $\ncap(0, 0)$ and ends with $\ncup(0, 0)$.
Compare this with the abstract
tensor notation that can be used to describe quantum invariants, cf.~\cite{Kauffman,BaKi}. 
\begin{figure}[H]
\includegraphics[width=.45\textwidth]{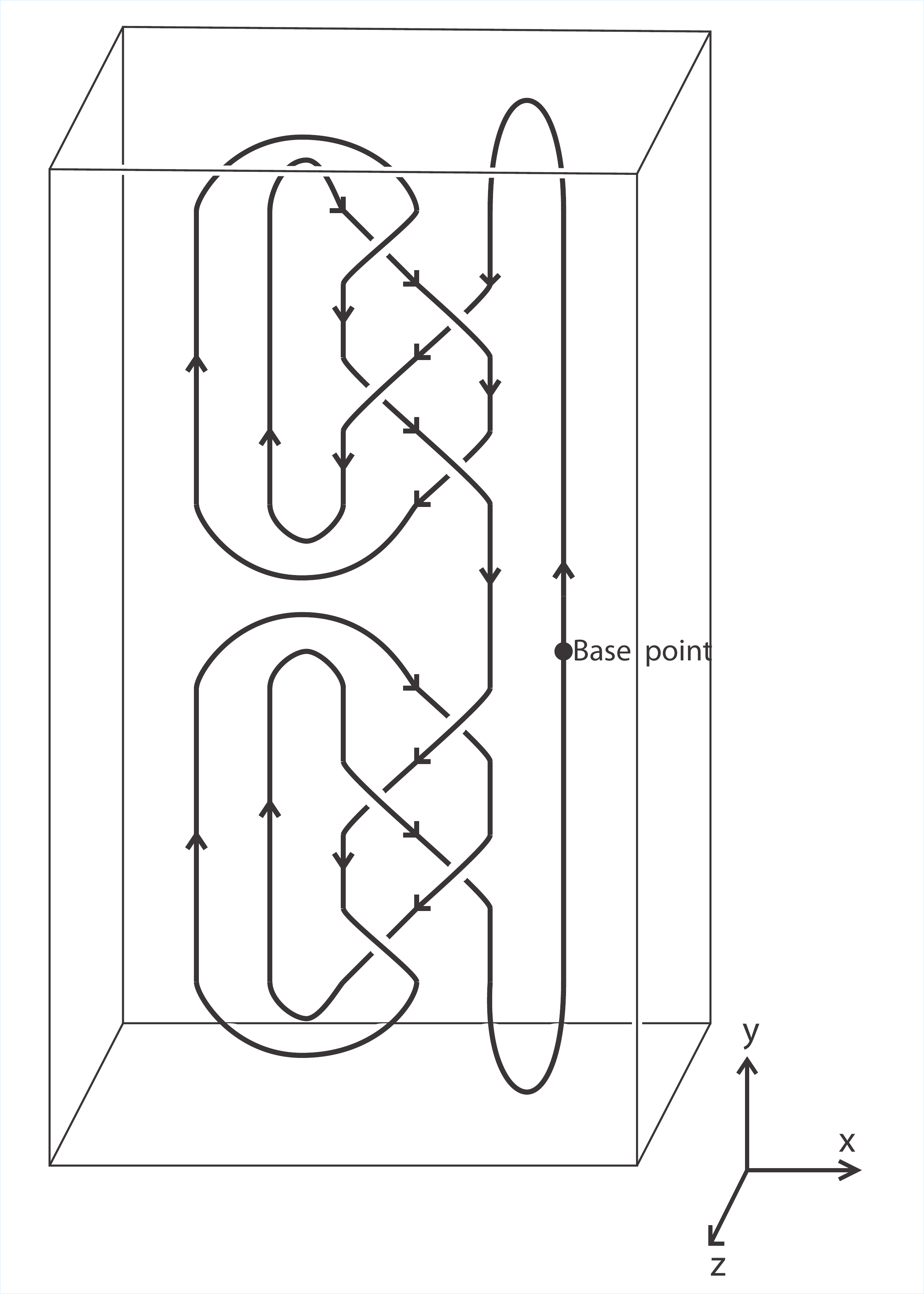}
\medskip
\caption{Connected sum of $4_1\# (-4_1)$}
\label{fig:H} 
\end{figure}
For example, the bookmark word for the knot in Figure \ref{fig:H} is 
\begin{figure}[H]
\includegraphics[height=1.3cm]{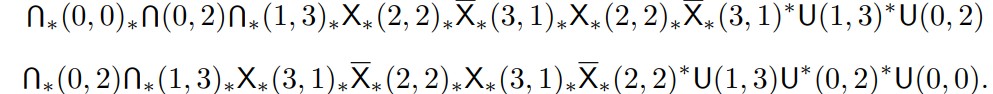}
\label{fig:bookmark} 
\end{figure}

\subsection{Definition and properties of charts}\label{sec:charts} 

A {\em (retinal) chart} is an oriented labeled graph contained in the $(s, y)$-plane. 
The structure of a chart is a Cerf-theoretic description, with additional decorations, of an immersed surface cobordism between classical links.
We stress that there is a chosen preferred direction, $y$, in the plane, and the chart is defined with respect to this choice. 
In the following subsections we define its edges, edge decorations, and the types of vertices, see \cite[1.5]{CSKnottedSurfaces}, \cite[Section 3.2]{CRS} for more details.

\subsubsection{The edges} 
The chart of a surface $F\looparrowright \R^3\times [0,1]$ has two types of edges, corresponding to folds and to
crossings as described next. The edges of the graph may cross. Such crossings are among the vertices of the chart.

According to singularity theory, a generic map from a surface to $\R^2$ has fold singularities. These form a 1-dimensional set upon which the rank of the map drops by 1. At cusps the rank of the projection map drops
to 0. The images of the folds in the $(s,y)$-plane form one type of edge of a chart. The second type of edge corresponds to crossings with respect to the $z$ coordinate: consider the double point arcs of the projection of $F$ to the 3-dimensional space with the $(x,y,s)$-coordinates, and further project them to the $(s,y)$-plane.

To relate the charts to the stills $F\cap (\R^3\times\{ s\})$, note that for generic values of $s$, the intersection of the vertical line $\{ s\}\times \R$ in the $(s,y)$-plane with the chart consists of a finite collection of points corresponding to critical points of the link $F\cap (\R^3\times\{ s\})$ with respect to the $y$ coordinate, and to the crossings of the link $F\cap (\R^3\times\{ s\})$. The $y$ coordinates of these points in the plane are exactly the same as the $y$ coordinates of these critical points, respectively crossings, of the link. 

\subsubsection{Edge decorations}
The fold edges and the crossing (or double point) edges are labeled by the corresponding symbols $\nx, \overline{\nx}, \ncap, \ncup$, and additionally they have a pair of non-negative integer labels. The
first integer indicates the number of surface sheets that are behind the line of
sight or to the left in the $x$-direction of it, and the second indicates the number
of sheets that occlude the fold in the line of sight or to the right of the fold or crossing 
in the $x$-direction. This pair of integers agrees with the bookmark code of a
cross-sectional still. Folds are also decorated with a short  upward or downward pointing arc that indicates the side of the fold at which the surface overlaps. In this way, saddles and optima can be distinguished in the chart. Specifically, at a birth or death these vertical arcs both point inward; at saddles they point outward.

The intersection of a vertical line segment, $s = s_i$, with the chart, that
does not pass through any vertices of the chart, results in the bookmark code for
the link $F \cap (\R^3 \times \{s_i\})$.

\subsubsection{The vertices} \label{subsubsec: vertices}

Consider the critical points (births, deaths, saddle points) of the projection onto the $s$-axis
\[ F\looparrowright \R^3 \times [0, 1] \overset{p}{\longrightarrow} [0,1]. \]
For terminological precision, these critical points and
the Reidemeister moves, cusps, $\psi$-moves, and crossing changes (nodes) will be
included among the {\em critical events}. Critical events are projected onto the $(s, y)$-plane, 
and they represent vertices in the chart that correspond to changes in the bookmark code.

\begin{enumerate}
    \item The critical points of the folds — births, deaths, and saddles — are valence two vertices of the chart graph. They are Morse singularities with respect to the $s$-direction, Figures \ref{fig:J}, \ref{fig:K}. The orientations of the folds at these junctures are inconsistent.    
\begin{figure}[H]
\includegraphics[width=.8\textwidth]{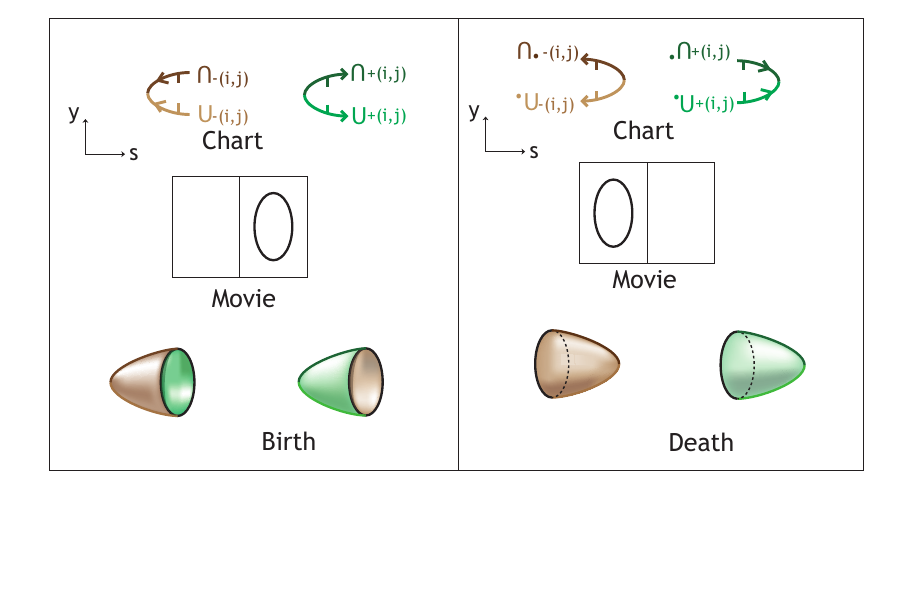}
\caption{Births and deaths}
\label{fig:J} 
\end{figure}
\begin{figure}[H]
\includegraphics[width=.8\textwidth]{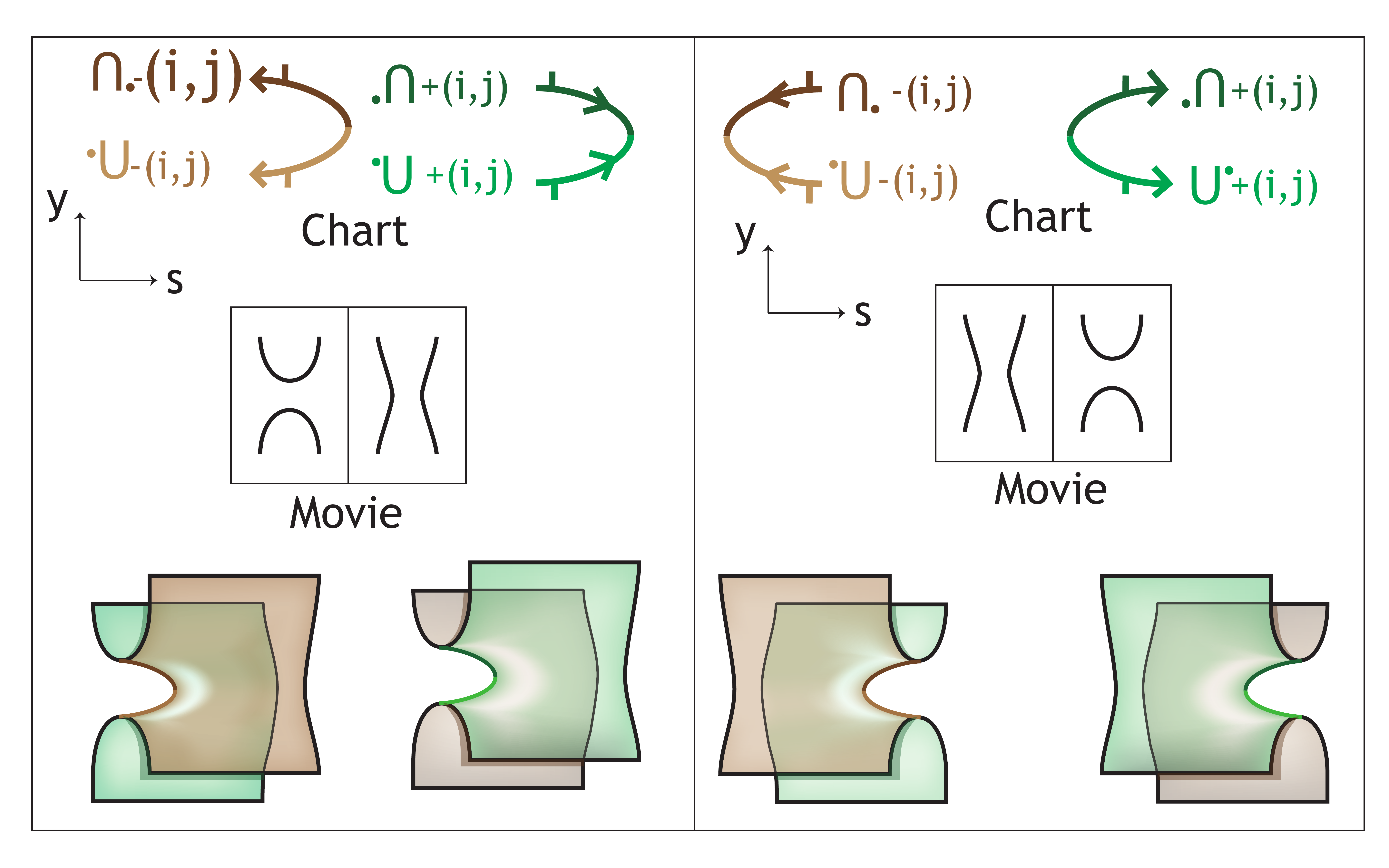} 
\caption{Saddles}
\label{fig:K}
\end{figure}

    \item Cusps are also valence two vertices at which folds are created or terminated. Both a green and a brown edge are incident at a cusp. The orientations of the folds are consistent at the cusp. There are eight types of cusps that depend upon directions, orientations and
bookmarks. Two are illustrated, Figure \ref{fig:L}. The reader is encouraged to catalogue the
remaining cases. See also \cite{CKDiagrammaticAlgebra}. 
    \begin{figure}[H]
\includegraphics[height=4cm]{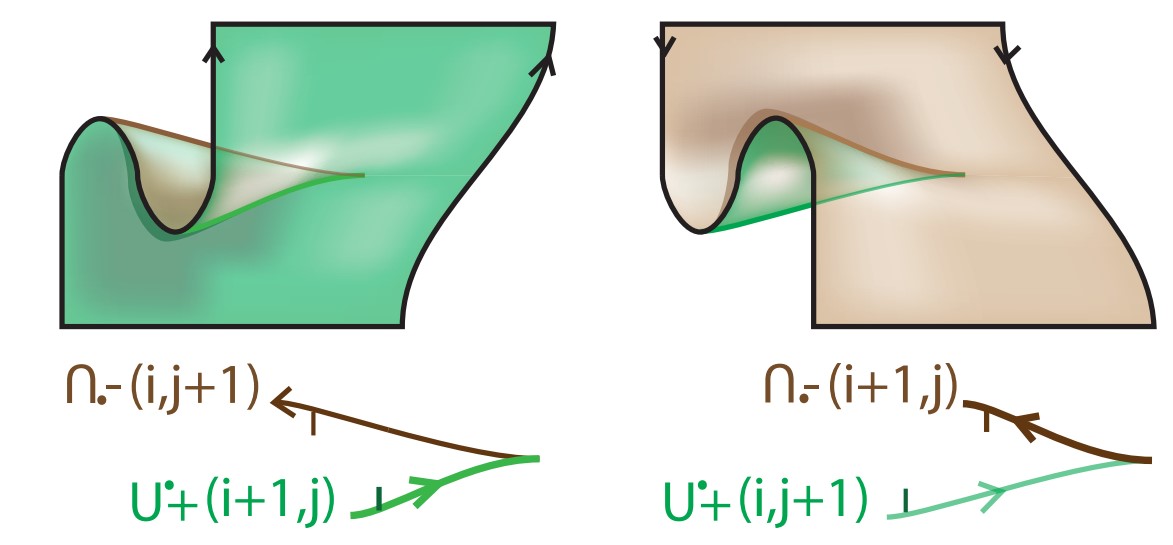} 
\caption{Examples of cusps and their corresponding chart vertices}
\label{fig:L}
\end{figure}

\item Valence three vertices that correspond to the branch points created by
type-I moves, are the junction of a green fold, a brown fold, and a crossing.
In drawing the graph, the fold set seems to run straight through at this
vertex. However, the orientation of the folds is inconsistent at a branch point. Branch points in the chart correspond to Reidemeister type-I moves
in a movie presentation, Figure \ref{fig:M}.
\begin{figure}[H]
\includegraphics[height=3.5cm]{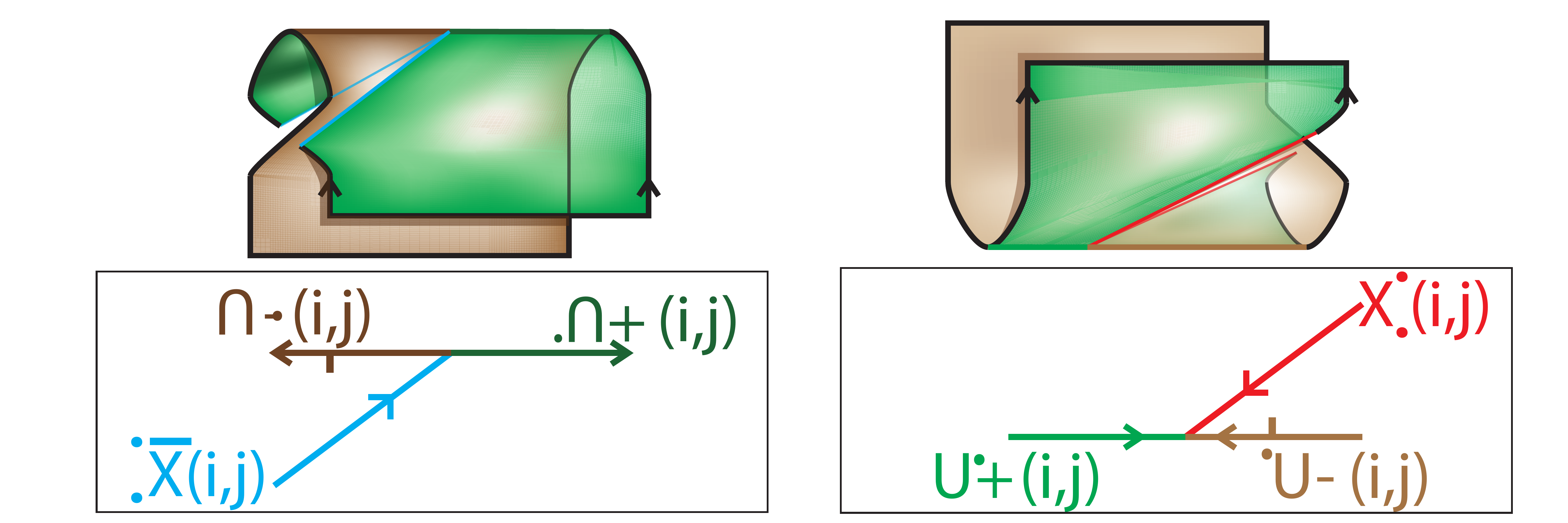}
\caption{Two examples of type-I chart vertices}
\label{fig:M} 
\end{figure}

\item Valence two vertices arising from type-II moves correspond to critical
points of the crossing points. The incident crossings are colored red and
blue, and the orientation flows through the vertex, Figure \ref{fig:N}.
\begin{figure}[H]
\includegraphics[height=4.5cm]{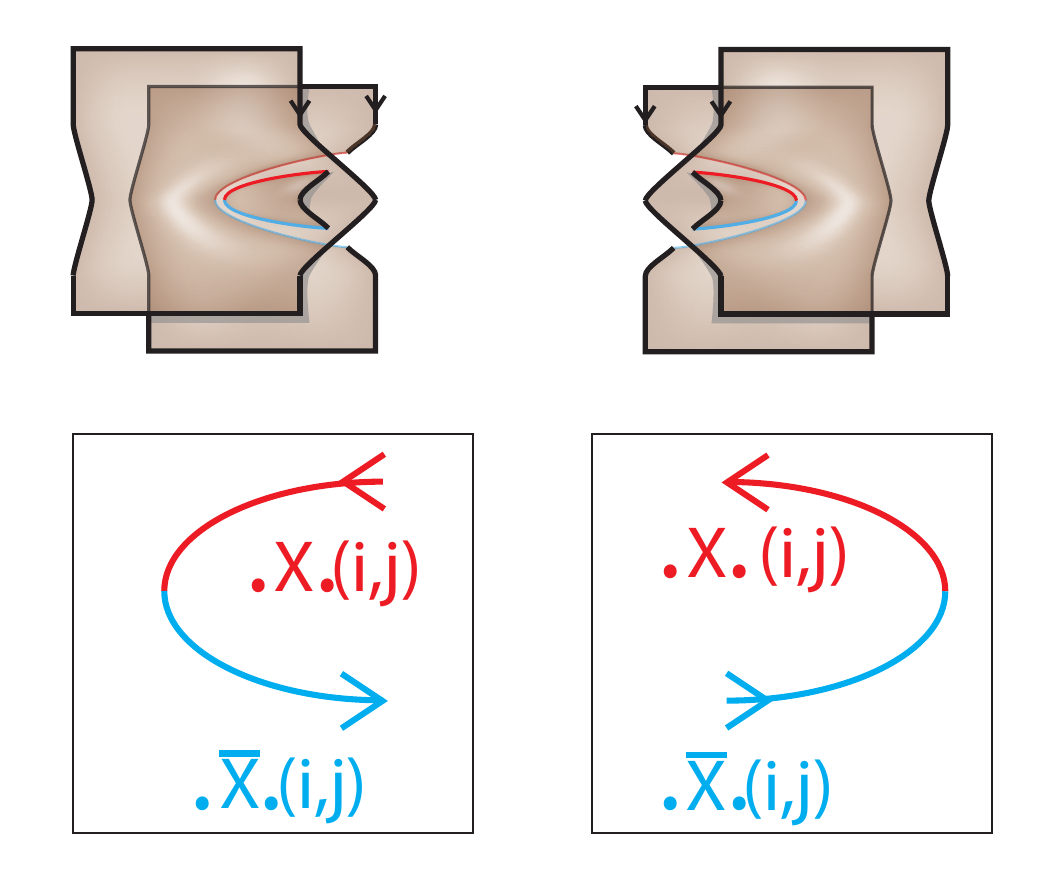}
\caption{Two examples of type-II chart vertices}
\label{fig:N} 
\end{figure}

\item Valence four vertices that indicate a crossing point passing over a fold
involve two crossings and two folds. Both pairs of edges are oriented
consistently. The bookmark codes change along the four incident edges, Figure \ref{fig:O}.
\begin{figure}[H]
\includegraphics[height=4cm]{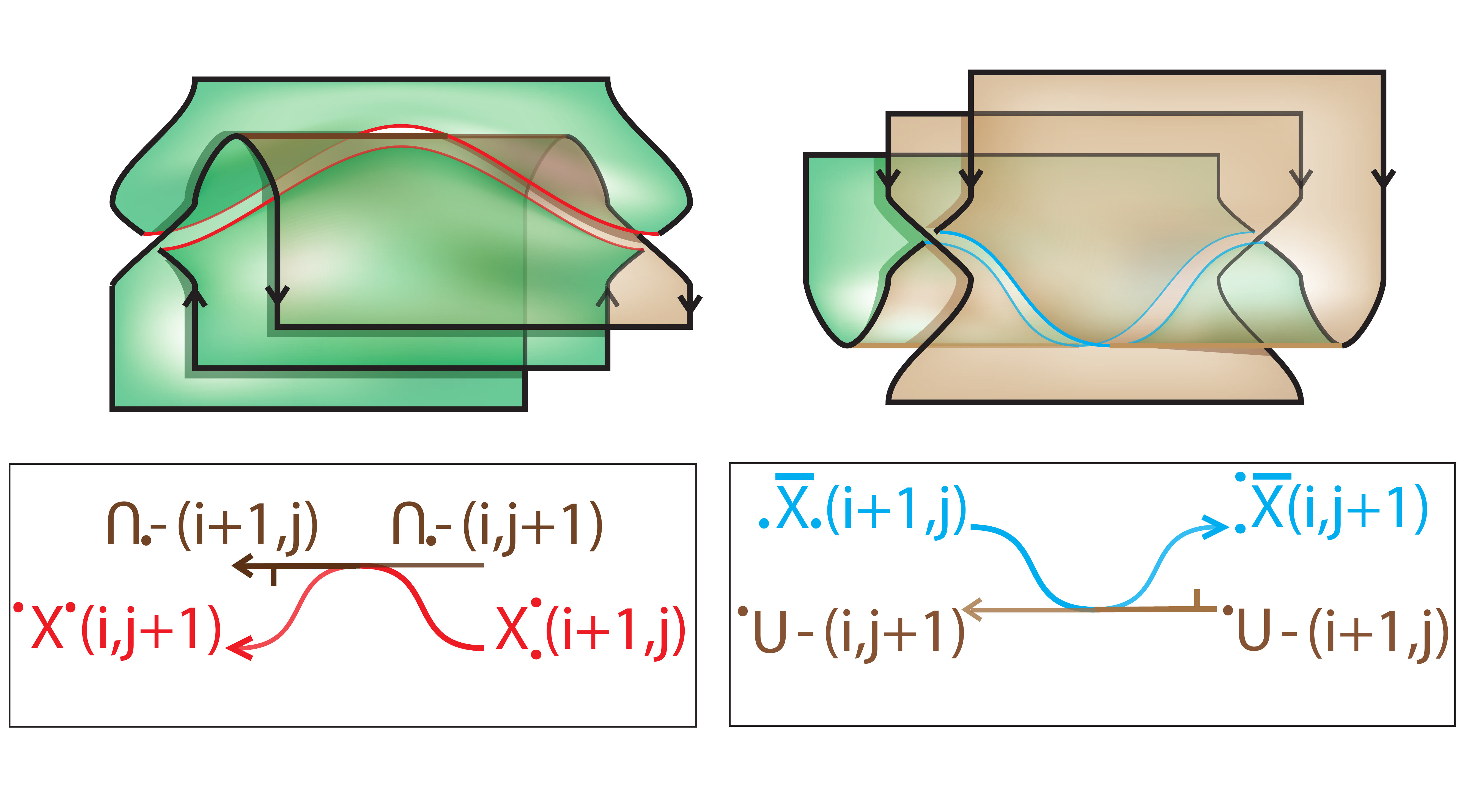}
\caption{Two examples  of $\psi$-move chart vertices}
\label{fig:O} 
\end{figure}

\item \label{item: typeIII} Valence six vertices correspond to triple points when surface F is projected
into the $(x, y, s)$  3-space. Recall that the $z$-coordinate is used for the crossing sense of the stills in a movie. These vertices correspond
to Reidemeister type-III moves.
\begin{figure}[H]
\includegraphics[height=4cm]{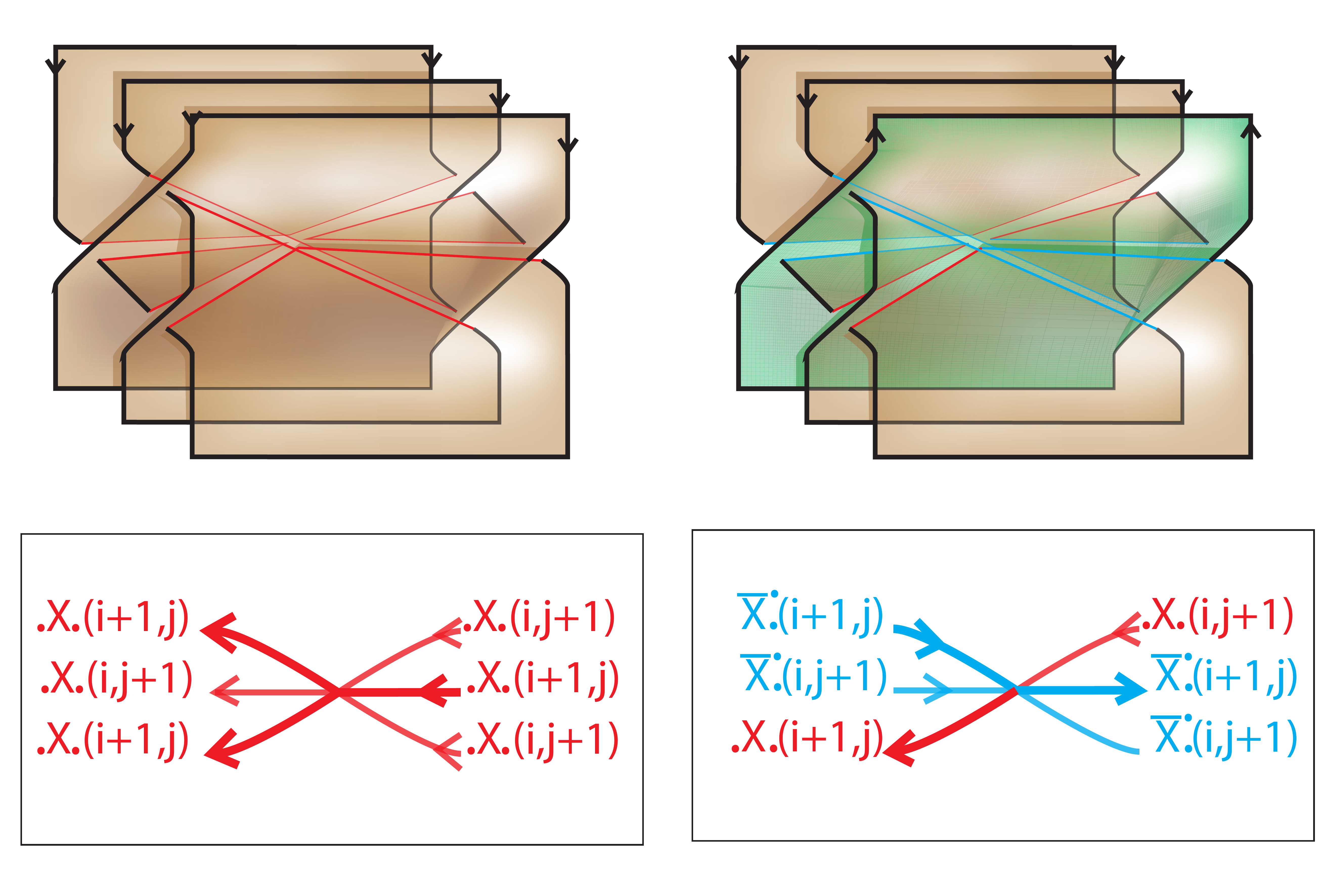}
\caption{Two examples of type-III chart vertices}
\label{fig:P} 
\end{figure}

\item There are valence two vertices that correspond to nodes. A source and a
sink are indicated.
\begin{figure}[H]
\includegraphics[height=4cm]{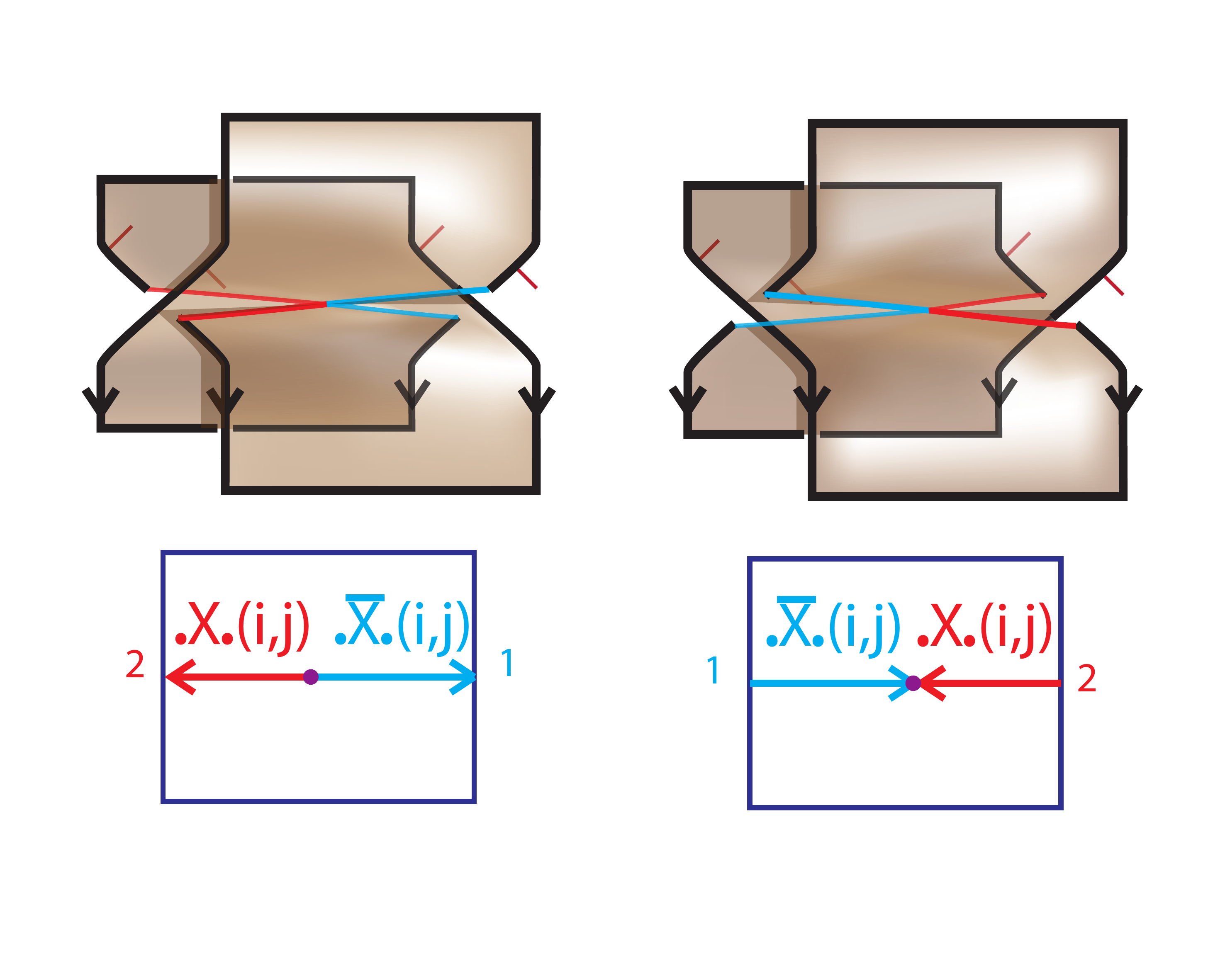}
\caption{A source node (left) and a sink (right)}
\label{fig:Q} 
\end{figure}

\end{enumerate}

Note that double point curves and folds that have disparate bookmark codes may also
cross. To understand them, think, for example, of the braid relation $\sigma_i\sigma_j = \sigma_j\sigma_i$
for $1 < |i - j|$. See Figure \ref{fig:D}.

\subsection{The movie move theorem}
We are in a position to formulate the main result of this section. The figures following the statement of the theorem include both the chart moves and the movie moves. 

\begin{theorem} \label{thm: movie moves}
The movie moves that describe isotopies of immersed surfaces are of the following types:
\begin{enumerate}
\item MM1-MM15: the Carter-Saito movie moves for embedded surfaces \cite[2.6]{CSKnottedSurfaces}
\item MM16: passing a node over a type-II move, Figure \ref{fig:B},
\item  MM17: passing a node through a type-III move, Figure \ref{fig:C},
\item MM18: passing a node over a fold, Figure \ref{fig:A},
\item   moves that correspond to sliding a node along a horizontal segment of a double point curve; see for example, Figures \ref{fig:D}-\ref{fig:G}. 
\end{enumerate}
\end{theorem}

In this theorem, the new moves that involve nodes of immersed surfaces are those of types (2)-(5).

There are variations of each of these moves that depend, for example,
upon differing orientations, directions of the type-II moves, or whether the fold
involved is a local maximum or minimum. The reader is encouraged to investigate these variations. We caution that nodes may not pass through the
bottom or top arc at a type-III move.

We have labeled the move in Fig.~\ref{fig:D} as MM19 since, in this case, the node interacts directly with a vertex of the chart even though the vertex is an interchange of distant crossings. 
Both Figures~\ref{fig:F} and~\ref{fig:G} demonstrate that a node commutes with a saddle point in the fold set. There are other, analogous types of commutation relations of this type: a node commutes with critical points of the fold set which could be a 
birth, death, or a cusp. A node also commutes with a type I, II, or III move if the double point curve upon which the node sits is not involved in the Reidemeister move. As explained in the proof below, all these cases can be understood systematically as the horizontal coordinate of the node interchanging with one of these chart vertices.

\begin{figure}[H]
\includegraphics[height=4cm]{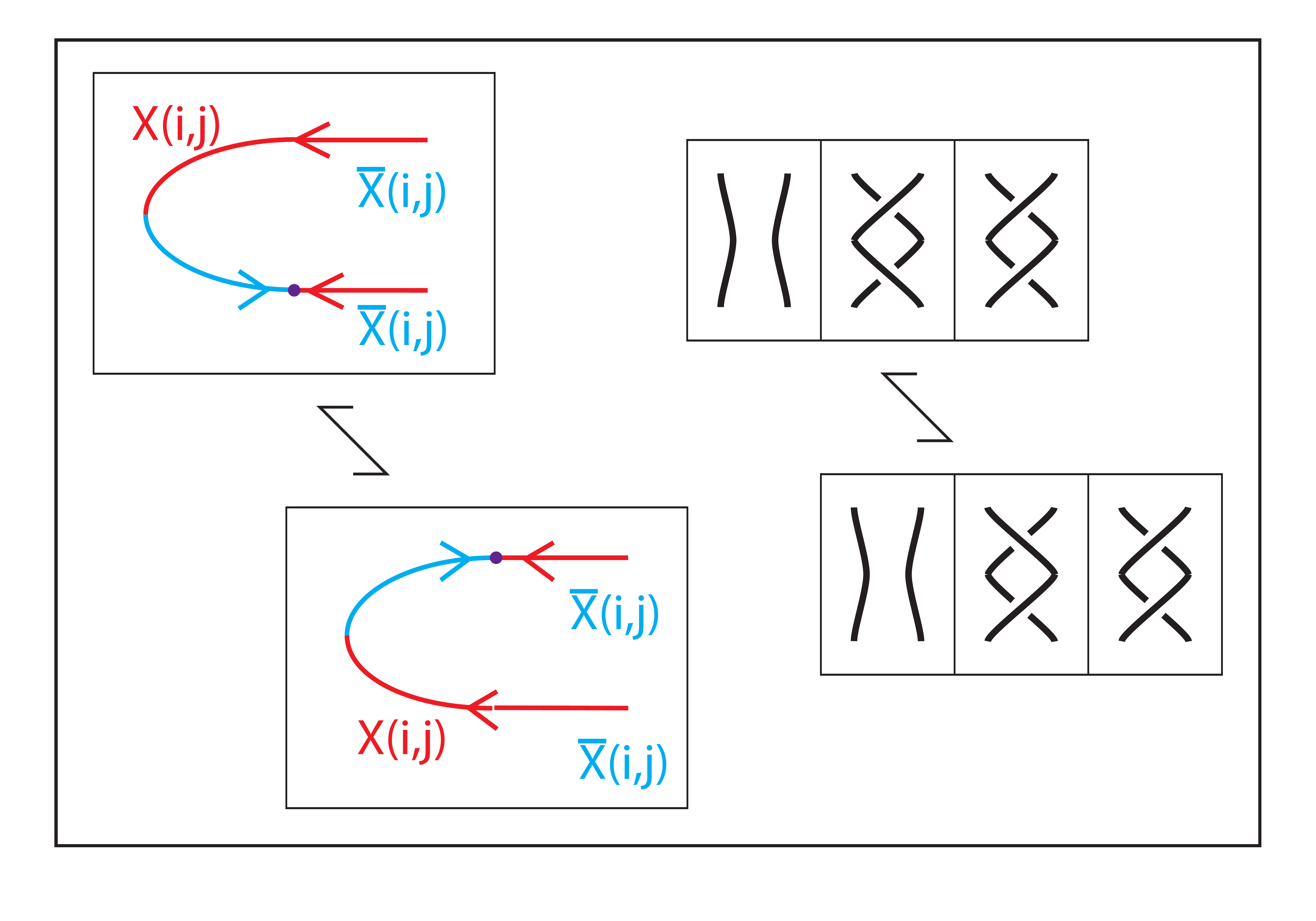}
\caption{MM16: passing a node over a type-II move.}
\label{fig:B} 
\end{figure}

\begin{figure}[H]
\includegraphics[height=4cm]{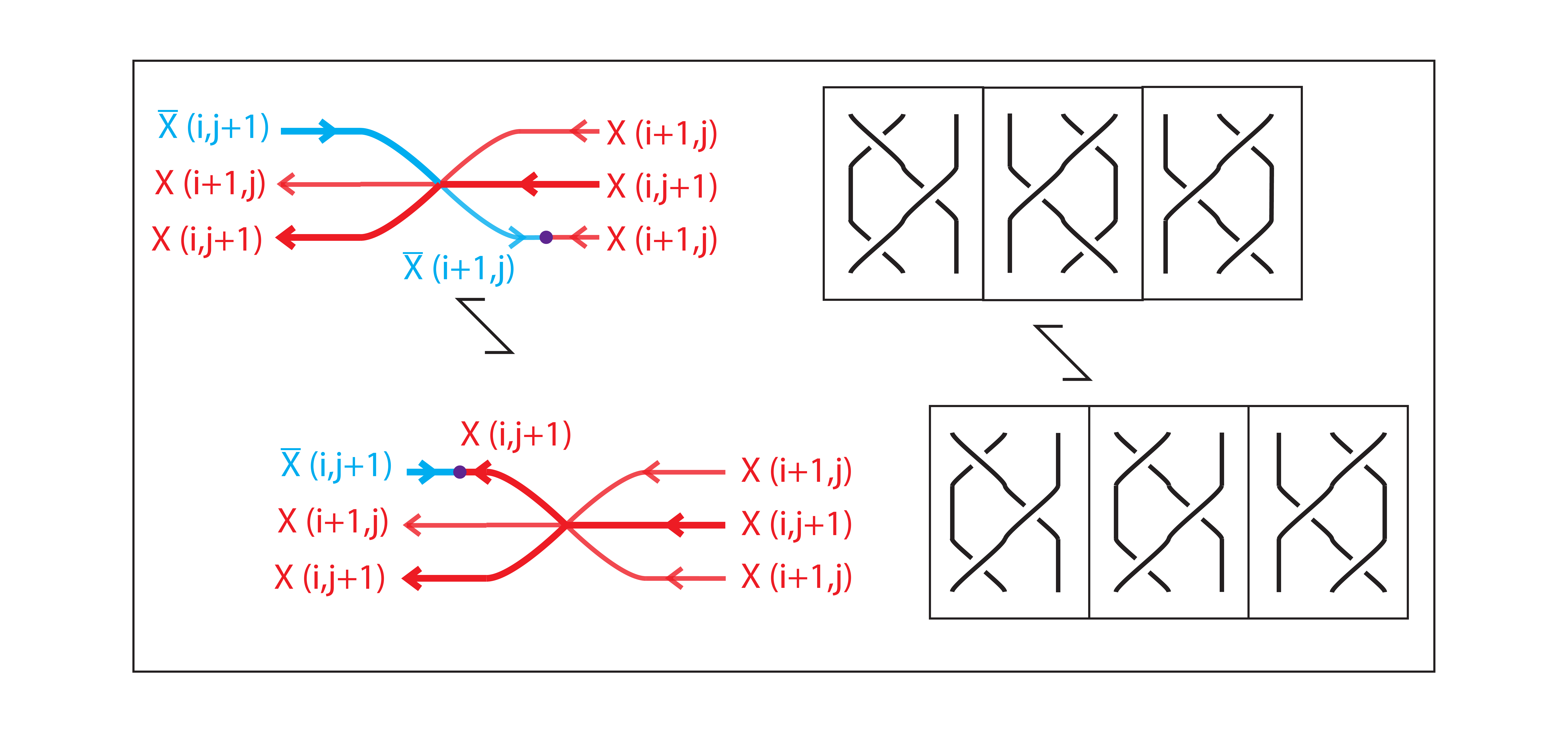}
\caption{MM17: moving a node through a triple point.}
\label{fig:C} 
\end{figure}

\begin{figure}[H] 
\includegraphics[height=5cm]{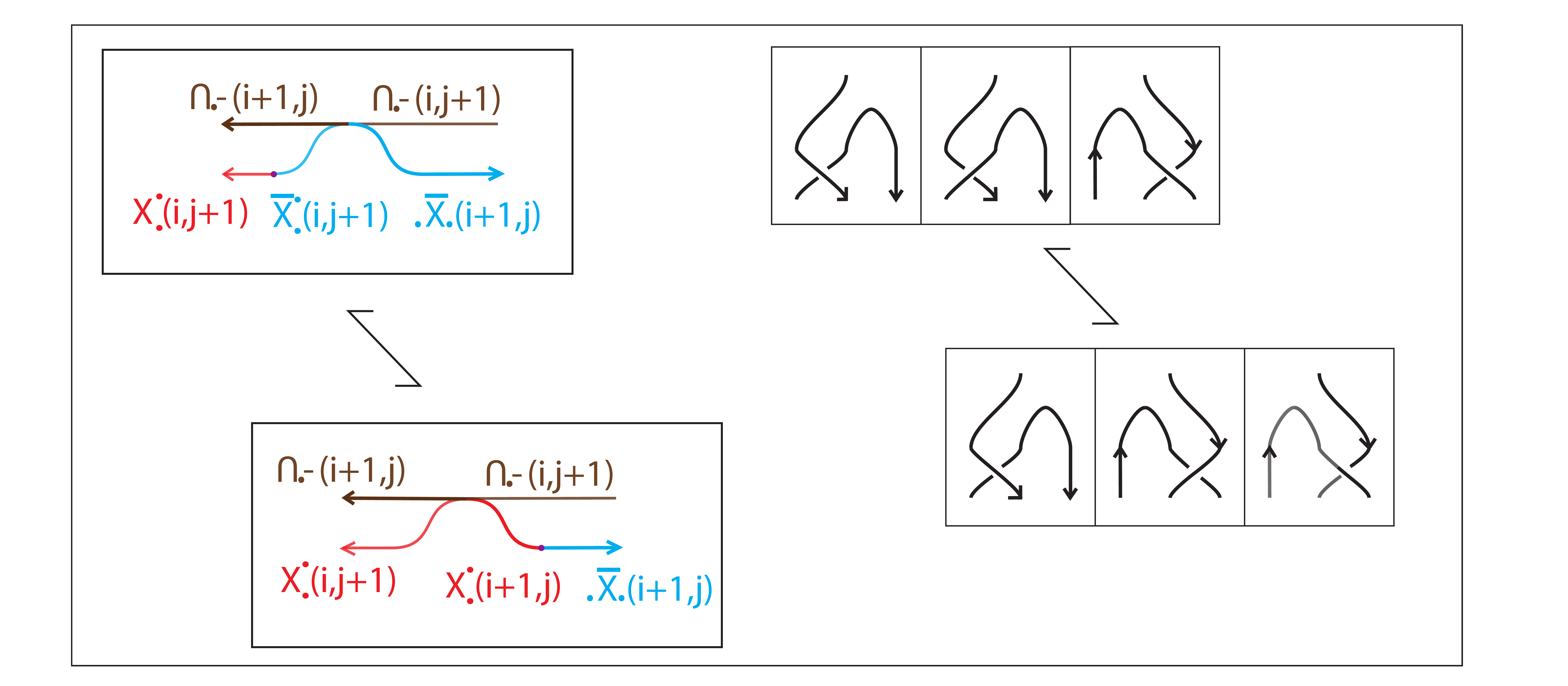}
\caption{MM18: moving a node over a fold.}
\label{fig:A} 
\end{figure}

\begin{figure}[H]
\includegraphics[height=3.5cm]{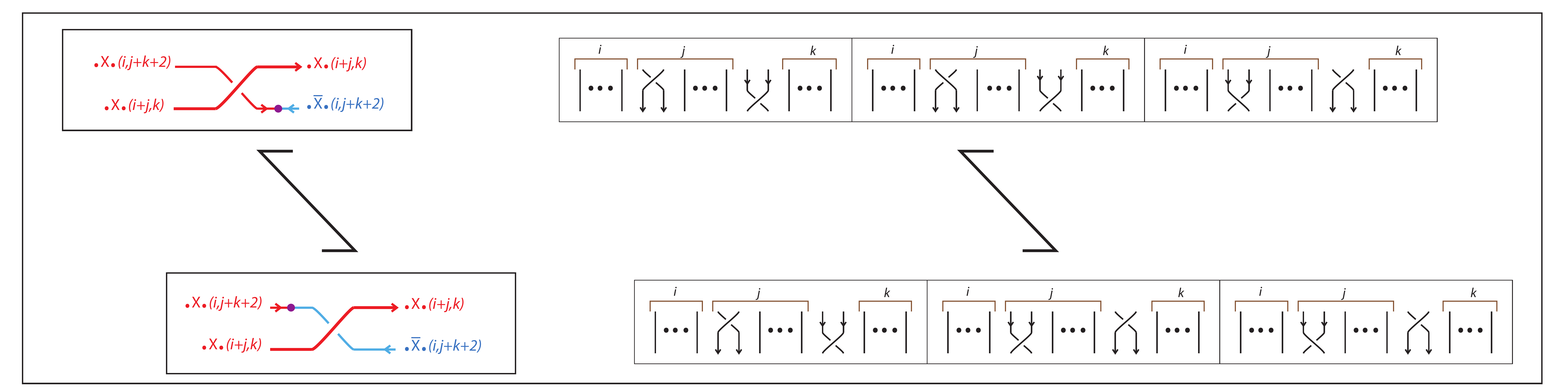}
\caption{MM19: Interchanging a node and a distant crossing.}
\label{fig:D} 
\end{figure}

\begin{figure}[H]
\includegraphics[height=4cm]{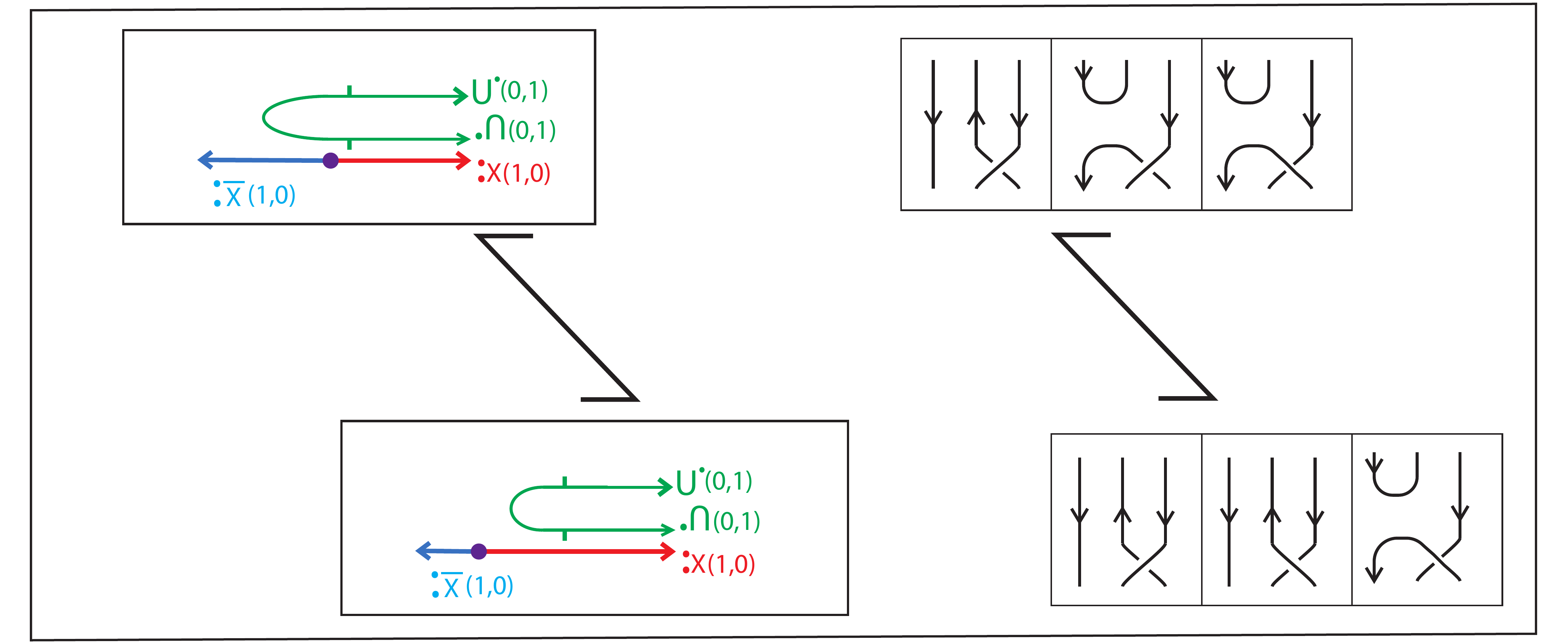} 
\caption{Commutation of a critical point and a node, version 1.}
\label{fig:F}
\end{figure}

\begin{figure}[H]
\includegraphics[height=5cm]{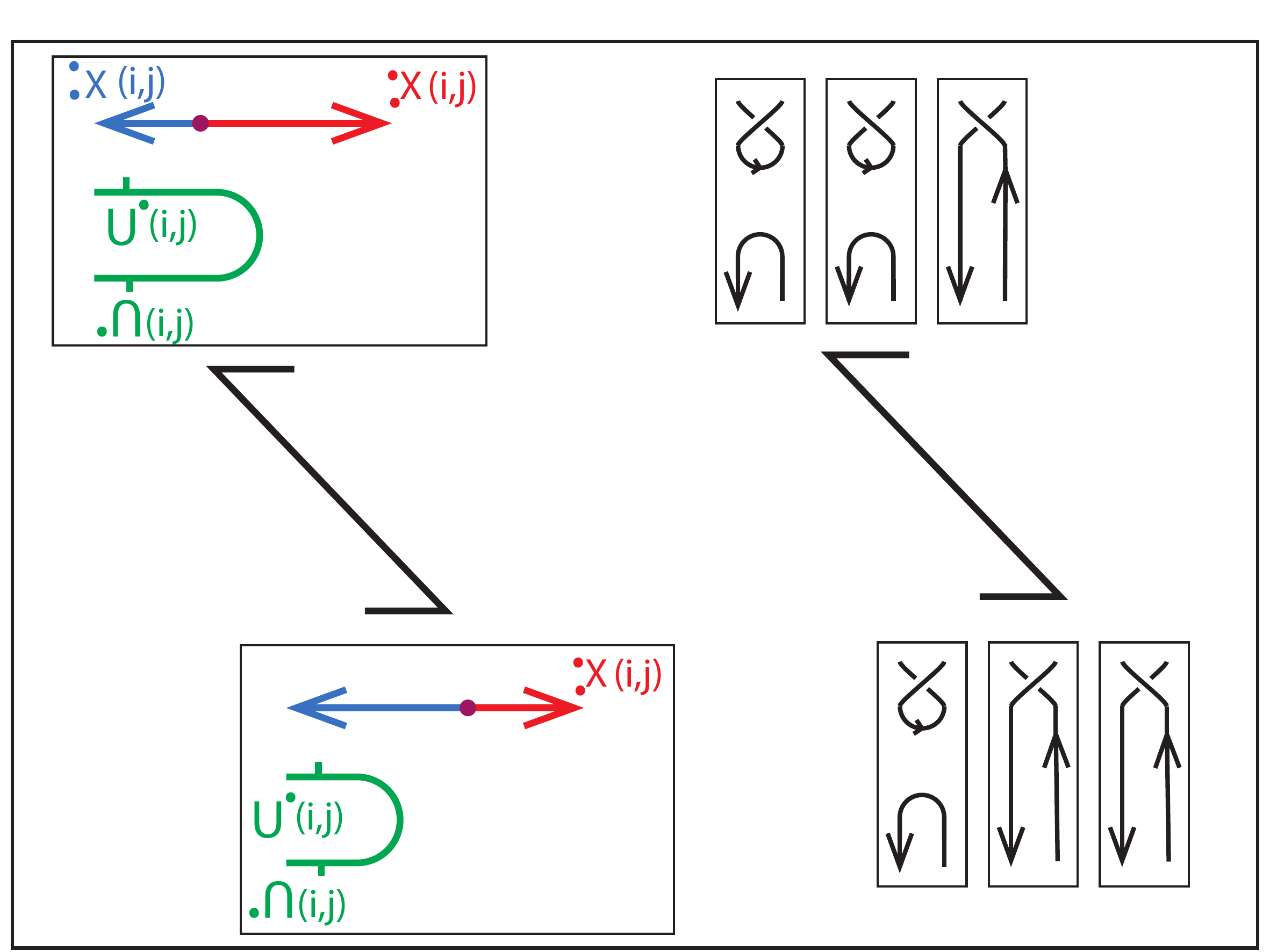} 
\caption{Commutation of a critical point and a node, version 2.}
\label{fig:G}
\end{figure}

\noindent
{\em Proof of Theorem \ref{thm: movie moves}.} The double point curves that appear
among the edges of a chart form a 1-dimensional manifold with boundary that
is immersed into the $(s, y)$-plane. The immersion is not proper nor is it in
general position. End points of double point arcs occur at the junction of folds
of valence 3 vertices — type-I moves.

Non-generic double points of the immersed double curves occur at the valence 6 vertices of the chart which represent triple points — type-III moves, see item  (\ref{item: typeIII}) and Figure \ref{fig:P} in Section \ref{subsubsec: vertices}.
While they are non-generic, they are transverse. Antipodal arcs at
a valence 6 vertex are components of the same double point arc. 

The nodes of the chart are a 0-dimensional subset of the 1-dimensional crossing set. A key observation is that during an isotopy of a surface with nodes, the nodes
can move only along the double point arcs. In the chart, double point arcs are
horizontal in the $(s,y)$-plane. In Figure \ref{fig:B} it appears that a node passes through a point of vertical tangency of a double point curve. However, the point of vertical tangency is considered a vertex of the chart graph. 
Generically, the vertices of the chart have distinct
$s$-coordinates. As a node slides along a double point arc, its $s$-coordinate can
interchange with the $s$ coordinate of another vertex. In Figure \ref{fig:D} above, the
node slides beyond the crossing of a pair of double curves. In Figures \ref{fig:F} and \ref{fig:G}
it interchanges $s$ coordinates with a saddle. 
As discussed after the statement of the theorem, there are other analogous kinds of interchanges.

The remaining moves that involve nodes occur when a node passes through
a vertex.  These occur precisely when a node passes through a
type-II or type-III move, or passes through a $\psi$-move. These are, respectively, MM16 (Figure \ref{fig:B}),
MM17 (Figure \ref{fig:C}) and 
MM18 (Figure \ref{fig:A}).

Note that if a node were to pass through a type-I move, then the node would disappear,
see Figure \ref{fig:R}. This is impossible in an isotopy of a surface (because the number of double points is always preserved).
\begin{figure}[H]
\includegraphics[height=5cm]{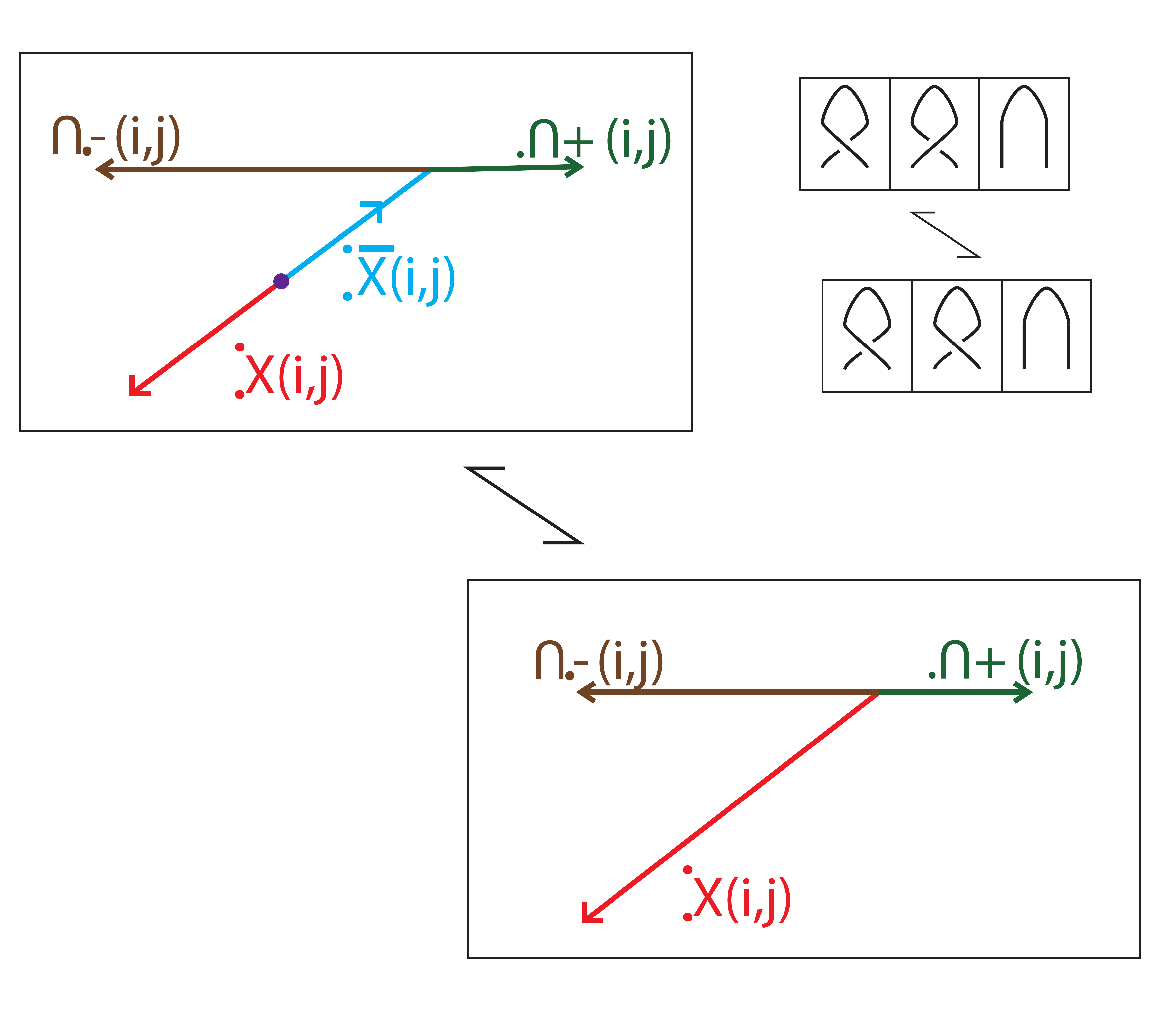} 
\caption{A node passing through a branch point changes the
number of nodes.}
\label{fig:R}
\end{figure}

This completes the proof of the theorem.
\qed

\bibliography{doublepoint}{}
\bibliographystyle{amsalpha}

\end{document}